\documentclass{amsart}
\usepackage[left=1.5in,right=1.5in]{geometry}
\usepackage{amsthm}
\usepackage{amsmath}
\usepackage{mathtools}
\usepackage{amssymb}
\usepackage{hyperref}
\usepackage{tikz}
\usetikzlibrary{decorations.markings}
\usepackage{enumitem}
\usepackage{amsfonts}
\usepackage{hyperref}
\usepackage{todonotes}
\usepackage[utf8]{inputenc}

\newcommand\precdot{\mathrel{\ooalign{$\prec$\cr
  \hidewidth\raise0.001ex\hbox{$\cdot\mkern0.6mu$}\cr}}}

\newtheorem{theorem}{Theorem}[section]
\newtheorem{prop}[theorem]{Proposition}

\newtheorem{lemma}[theorem]{Lemma}
\newtheorem{cor}[theorem]{Corollary}

\theoremstyle{definition}
\newtheorem{ex}[theorem]{Example}

\newtheorem{prob}[theorem]{Problem}
\newtheorem{defin}[theorem]{Definition}

\theoremstyle{remark}
\newtheorem*{remark}{Remark}

\DeclareMathOperator{\supp}{Supp}

\newcommand{\mc}{\mathcal}
\renewcommand{\tilde}{\widetilde}
\newcommand{\Inv}{\mathrm{Inv}}

\title{Minimal elements for the limit weak order on affine Weyl groups}

\author{Christian Gaetz}
\thanks{C.G. is supported by a National Science Foundation Graduate Research Fellowship under Grant No. 1122374.}
\address{Department of Mathematics, Massachusetts Institute of Technology, Cambridge, MA 02139}
\email{\href{mailto:gaetz@mit.edu}{{\tt gaetz@mit.edu}}}
\author{Yibo Gao}
\email{\href{mailto:gaoyibo@mit.edu}{{\tt gaoyibo@mit.edu}}}

\date{\today}

\begin{document}
\begin{abstract}
The limit weak order on an affine Weyl group was introduced by Lam and Pylyavskyy \cite{lam-pylyavskyy} in their study of total positivity for loop groups \cite{lam-pylyavskyy-part1}.  They showed that in the case of the affine symmetric group the minimal elements of this poset coincide with the infinite fully commutative reduced words and with infinite powers of Coxeter elements.  We answer several open problems raised there by classifying minimal elements in all affine types and relating these elements to the classes of fully commutative and Coxeter elements. Interestingly, the infinite fully commutative elements correspond to the minuscule and cominuscule nodes of the Dynkin diagram, while the infinite Coxeter elements correspond to a single node, which we call the \emph{heavy} node, in all affine types other than type $A$.
\end{abstract}

\maketitle

\section{Introduction}
\label{sec:intro}

An infinite word $s_{i_1}s_{i_2}\cdots $ in the simple generators of an infinite Coxeter group $\tilde{W}$ (see Section~\ref{sec:Coxeter-background} for background) is called an \emph{infinite reduced word} if all of its finite prefixes $s_{i_1}\cdots s_{i_k}$ are reduced words in $\tilde{W}$; we will identify such a word with its sequence $\mathbf{i}=i_1 i_2 \cdots$ of indices.  Associated to $\mathbf{i}$ is an \emph{inversion set} $\Inv(\mathbf{i})$, a subset of the set of reflections of $\tilde{W}$, which induces an equivalence relation on the set of infinite reduced words: $[\mathbf{i}]=[\mathbf{j}]$ if and only if $\Inv(\mathbf{i})=\Inv(\mathbf{j})$.  An equivalence class of infinite reduced words is called a \emph{limit element} of $\tilde{W}$.

The \emph{limit weak order} for $\tilde{W}$, introduced by Lam and Pylyavskyy \cite{lam-pylyavskyy} is the partial order $(\tilde{\mc{W}},\leq)$ on the set of limit elements, with order given by containment of inversion sets.  In \cite{lam-pylyavskyy} this order (conjecturally) encodes the containment relations between certain strata in the totally positive space studied there, while Lam and Thomas show in \cite{lam-thomas} that $\tilde{\mc{W}}$ encodes the closure relations among components of the Tits boundary of $\tilde{W}$.  In both instances, understanding the minimal elements in the limit weak order is of significant interest.  In the case when $\tilde{W}$ is the affine symmetric group, Lam and Pylyavskyy show that the minimal elements coincide with two other important classes of elements: fully commutative limit elements and infinite powers of Coxeter elements.

\begin{theorem}[Lam and Pylyavskyy \cite{lam-pylyavskyy}]
\label{thm:type-A}
Let $\tilde{W}$ be the affine symmetric group, then the following are equivalent for an infinite reduced word $\mathbf{i}$:
\begin{enumerate}
    \item $[\mathbf{i}]$ is minimal in $\tilde{\mc{W}}$,
    \item $[\mathbf{i}]$ is fully commutative,
    \item $[\mathbf{i}]=[c^{\infty}]$ for a Coxeter element $c$ of $\tilde{W}$.
\end{enumerate}
\end{theorem}

As natural extensions of Theorem~\ref{thm:type-A}, Lam and Pylyavskyy posed the following open problems:

\begin{prob}[Lam and Pylyavskyy \cite{lam-pylyavskyy}]
\label{prob:describe-words}
Describe, in terms of infinite reduced words, the minimal elements in limit weak order for all affine Weyl groups.
\end{prob}

\begin{prob}[Lam and Pylyavskyy \cite{lam-pylyavskyy}]
\label{prob:are-minimal-fully-commutative}
Are all minimal elements in limit weak order fully commutative?
\end{prob}

In this paper\footnote{An extended abstract of this work has been submitted to the proceedings of FPSAC 2021.} we describe complete resolutions of Problems~\ref{prob:describe-words} and \ref{prob:are-minimal-fully-commutative} and give further extensions of Theorem~\ref{thm:type-A}:  
\begin{itemize}
    \item In Section~\ref{sec:background} we cover needed background material.
    \item In Section~\ref{sec:translation} we note that, when $\tilde{W}$ is an affine Weyl group with corresponding finite Weyl group $W$, the minimal elements of $\tilde{\mc{W}}$ coincide with infinite powers of translations by multiples of $W$-conjugates of fundamental coweights.  We give a general, type-uniform procedure for generating infinite reduced words corresponding to these elements; the resulting classification of minimal elements, which resolves Problem~\ref{prob:describe-words}, is given in Table~\ref{tab:all-minimal}.
    \item Although none of the three equivalences in Theorem~\ref{thm:type-A} continues to hold in general affine Weyl groups, we show that infinite fully commutative elements and infinite powers of Coxeter elements are still minimal in $\tilde{\mc{W}}$.  Therefore it makes sense to ask for which fundamental coweights $\omega_i^{\vee}$ the corresponding infinite translation element is fully commutative or is a power of a Coxeter element; the answer is depicted in Figure~\ref{fig:main-fig}.  
    \item In fact we show in Section~\ref{sec:infinite-coxeter} that, except in type $A$, there is a unique $\omega_i^{\vee}$ corresponding to the Coxeter elements, and we give a simple rule for identifying the corresponding node in the Dynkin diagram, which we call the \emph{heavy node}. We also give a description of the heavy node in terms of the action of a bipartite Coxeter element on the highest root which may be of independent interest.
    \item In Section~\ref{sec:fully-commutative} we give a uniform proof that the fundamental coweights whose infinite translation elements are fully commutative are exactly those which are minuscule or cominuscule.  In particular, Problem~\ref{prob:are-minimal-fully-commutative} has a negative answer except in type $A$.
    \item Finally, Section~\ref{sec:densities} gives an alternative argument for the classification of infinite fully commutative words in terms of the ``densities" of generators $s_i$ appearing in the word.
\end{itemize}  
 
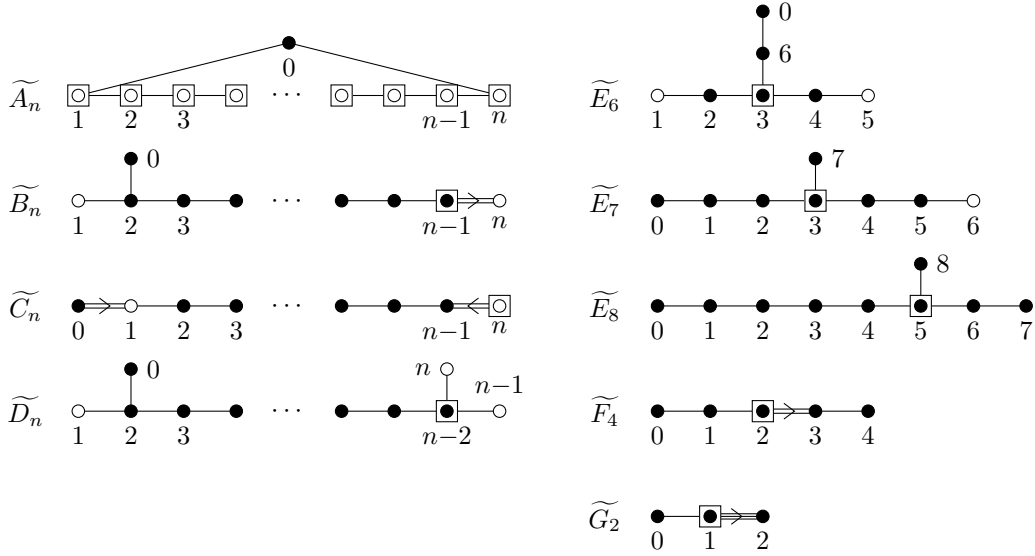
\begin{figure}
\centering
\begin{tikzpicture}[scale=0.7]
\draw(3,0)--(0,0)--(4,1)--(8,0)--(5,0);
\node[draw,shape=rectangle,fill=white,scale=1.2] at (0,0) {};
\node[draw,shape=rectangle,fill=white,scale=1.2] at (1,0) {};
\node[draw,shape=rectangle,fill=white,scale=1.2] at (2,0) {};
\node[draw,shape=rectangle,fill=white,scale=1.2] at (3,0) {};
\node[draw,shape=rectangle,fill=white,scale=1.2] at (5,0) {};
\node[draw,shape=rectangle,fill=white,scale=1.2] at (6,0) {};
\node[draw,shape=rectangle,fill=white,scale=1.2] at (7,0) {};
\node[draw,shape=rectangle,fill=white,scale=1.2] at (8,0) {};
\node[draw,shape=circle,fill=white,scale=0.5][label=below: {$1$}] at (0,0) {};
\node[draw,shape=circle,fill=white,scale=0.5][label=below: {$2$}] at (1,0) {};
\node[draw,shape=circle,fill=white,scale=0.5][label=below: {$3$}] at (2,0) {};
\node[draw,shape=circle,fill=white,scale=0.5][label=below: {}] at (3,0) {};
\node[draw,shape=circle,fill=white,scale=0.5][label=below: {}] at (5,0) {};
\node[draw,shape=circle,fill=white,scale=0.5][label=below: {}] at (6,0) {};
\node[draw,shape=circle,fill=white,scale=0.5][label=below: {$n{-}1$}] at (7,0) {};
\node[draw,shape=circle,fill=white,scale=0.5][label=below: {$n$}] at (8,0) {};
\node[draw,shape=circle,fill=black,scale=0.5][label=below: {$0$}] at (4,1) {};
\node at (4,0) {$\cdots$};
\node at (-1,0) {$\widetilde{A_n}$};

\draw(0,-2)--(3,-2);
\draw(5,-2)--(7,-2);
\draw(7,-2.04)--(8,-2.04);
\draw(7,-1.96)--(8,-1.96);
\draw(1,-1.2)--(1,-2);
\node[draw,shape=circle,fill=white,scale=0.5][label=below: {$1$}] at (0,-2) {};
\node[draw,shape=circle,fill=black,scale=0.5][label=below: {$2$}] at (1,-2) {};
\node[draw,shape=circle,fill=black,scale=0.5][label=below: {$3$}] at (2,-2) {};
\node[draw,shape=circle,fill=black,scale=0.5][label=below: {}] at (3,-2) {};
\node[draw,shape=circle,fill=black,scale=0.5][label=below: {}] at (5,-2) {};
\node[draw,shape=circle,fill=black,scale=0.5][label=below: {}] at (6,-2) {};
\node[draw,shape=rectangle,fill=white,scale=1.2] at (7,-2) {};
\node[draw,shape=circle,fill=black,scale=0.5][label=below: {$n{-}1$}] at (7,-2) {};
\node[draw,shape=circle,fill=white,scale=0.5][label=below: {$n$}] at (8,-2) {};
\node[draw,shape=circle,fill=black,scale=0.5][label=right: {$0$}] at (1,-1.2) {};
\draw(7.4,-1.85)--(7.6,-2)--(7.4,-2.15);
\node at (4,-2) {$\cdots$};
\node at (-1,-2) {$\widetilde{B_n}$};

\draw(1,-4)--(3,-4);
\draw(5,-4)--(7,-4);
\draw(0,-4.04)--(1,-4.04);
\draw(0,-3.96)--(1,-3.96);
\draw(7,-4.04)--(8,-4.04);
\draw(7,-3.96)--(8,-3.96);
\node at (4,-4) {$\cdots$};
\node[draw,shape=circle,fill=black,scale=0.5][label=below: {$0$}] at (0,-4) {};
\node[draw,shape=circle,fill=white,scale=0.5][label=below: {$1$}] at (1,-4) {};
\node[draw,shape=circle,fill=black,scale=0.5][label=below: {$2$}] at (2,-4) {};
\node[draw,shape=circle,fill=black,scale=0.5][label=below: {$3$}] at (3,-4) {};
\node[draw,shape=circle,fill=black,scale=0.5][label=below: {}] at (5,-4) {};
\node[draw,shape=circle,fill=black,scale=0.5][label=below: {}] at (6,-4) {};
\node[draw,shape=circle,fill=black,scale=0.5][label=below: {$n{-}1$}] at (7,-4) {};
\node[draw,shape=rectangle,fill=white,scale=1.2] at (8,-4) {};
\node[draw,shape=circle,fill=white,scale=0.5][label=below: {$n$}] at (8,-4) {};
\draw(7.6,-3.85)--(7.4,-4)--(7.6,-4.15);
\draw(0.4,-3.85)--(0.6,-4)--(0.4,-4.15);
\node at (-1,-4) {$\widetilde{C_n}$};

\draw(0,-6)--(3,-6);
\draw(1,-5.2)--(1,-6);
\draw(5,-6)--(8,-6);
\draw(7,-5.2)--(7,-6);
\node at (4,-6) {$\cdots$};
\node[draw,shape=circle,fill=black,scale=0.5][label=right: {$0$}] at (1,-5.2) {};
\node[draw,shape=circle,fill=white,scale=0.5][label=below: {$1$}] at (0,-6) {};
\node[draw,shape=circle,fill=black,scale=0.5][label=below: {$2$}] at (1,-6) {};
\node[draw,shape=circle,fill=black,scale=0.5][label=below: {$3$}] at (2,-6) {};
\node[draw,shape=circle,fill=black,scale=0.5][label=below: {}] at (3,-6) {};
\node[draw,shape=circle,fill=black,scale=0.5][label=below: {}] at (5,-6) {};
\node[draw,shape=circle,fill=black,scale=0.5][label=below: {}] at (6,-6) {};
\node[draw,shape=rectangle,fill=white,scale=1.2] at (7,-6) {};
\node[draw,shape=circle,fill=black,scale=0.5][label=below: {$n{-}2$}] at (7,-6) {};
\node[draw,shape=circle,fill=white,scale=0.5][label=above: {$n{-}1$}] at (8,-6) {};
\node[draw,shape=circle,fill=white,scale=0.5][label=left: {$n$}] at (7,-5.2) {};
\node at (-1,-6) {$\widetilde{D_n}$};

\draw(11,0)--(15,0);
\draw(13,0)--(13,1.6);
\node[draw,shape=circle,fill=white,scale=0.5][label=below: {$1$}] at (11,0) {};
\node[draw,shape=circle,fill=black,scale=0.5][label=below: {$2$}] at (12,0) {};
\node[draw,shape=rectangle,fill=white,scale=1.2] at (13,0) {};
\node[draw,shape=circle,fill=black,scale=0.5][label=below: {$3$}] at (13,0) {};
\node[draw,shape=circle,fill=black,scale=0.5][label=below: {$4$}] at (14,0) {};
\node[draw,shape=circle,fill=white,scale=0.5][label=below: {$5$}] at (15,0) {};
\node[draw,shape=circle,fill=black,scale=0.5][label=right: {$6$}] at (13,0.8) {};
\node[draw,shape=circle,fill=black,scale=0.5][label=right: {$0$}] at (13,1.6) {};
\node at (10,0) {$\widetilde{E_6}$};

\draw(11,-2)--(17,-2);
\draw(14,-2)--(14,-1.2);
\node[draw,shape=circle,fill=black,scale=0.5][label=below: {$0$}] at (11,-2) {};
\node[draw,shape=circle,fill=black,scale=0.5][label=below: {$1$}] at (12,-2) {};
\node[draw,shape=circle,fill=black,scale=0.5][label=below: {$2$}] at (13,-2) {};
\node[draw,shape=rectangle,fill=white,scale=1.2] at (14,-2) {};
\node[draw,shape=circle,fill=black,scale=0.5][label=below: {$3$}] at (14,-2) {};
\node[draw,shape=circle,fill=black,scale=0.5][label=below: {$4$}] at (15,-2) {};
\node[draw,shape=circle,fill=black,scale=0.5][label=below: {$5$}] at (16,-2) {};
\node[draw,shape=circle,fill=white,scale=0.5][label=below: {$6$}] at (17,-2) {};
\node[draw,shape=circle,fill=black,scale=0.5][label=right: {$7$}] at (14,-1.2) {};
\node at (10,-2) {$\widetilde{E_7}$};

\draw(11,-4)--(18,-4);
\draw(16,-4)--(16,-3.2);
\node[draw,shape=circle,fill=black,scale=0.5][label=below: {$0$}] at (11,-4) {};
\node[draw,shape=circle,fill=black,scale=0.5][label=below: {$1$}] at (12,-4) {};
\node[draw,shape=circle,fill=black,scale=0.5][label=below: {$2$}] at (13,-4) {};
\node[draw,shape=circle,fill=black,scale=0.5][label=below: {$3$}] at (14,-4) {};
\node[draw,shape=circle,fill=black,scale=0.5][label=below: {$4$}] at (15,-4) {};
\node[draw,shape=rectangle,fill=white,scale=1.2] at (16,-4) {};
\node[draw,shape=circle,fill=black,scale=0.5][label=below: {$5$}] at (16,-4) {};
\node[draw,shape=circle,fill=black,scale=0.5][label=below: {$6$}] at (17,-4) {};
\node[draw,shape=circle,fill=black,scale=0.5][label=below: {$7$}] at (18,-4) {};
\node[draw,shape=circle,fill=black,scale=0.5][label=right: {$8$}] at (16,-3.2) {};
\node at (10,-4) {$\widetilde{E_8}$};

\draw(11,-6)--(13,-6);
\draw(14,-6)--(15,-6);
\draw(13,-6.04)--(14,-6.04);
\draw(13,-5.96)--(14,-5.96);
\node[draw,shape=circle,fill=black,scale=0.5][label=below: {$0$}] at (11,-6) {};
\node[draw,shape=circle,fill=black,scale=0.5][label=below: {$1$}] at (12,-6) {};
\node[draw,shape=rectangle,fill=white,scale=1.2] at (13,-6) {};
\node[draw,shape=circle,fill=black,scale=0.5][label=below: {$2$}] at (13,-6) {};
\node[draw,shape=circle,fill=black,scale=0.5][label=below: {$3$}] at (14,-6) {};
\node[draw,shape=circle,fill=black,scale=0.5][label=below: {$4$}] at (15,-6) {};
\draw(13.4,-5.85)--(13.6,-6)--(13.4,-6.15);
\node at (10,-6) {$\widetilde{F_4}$};

\draw(11,-8)--(13,-8);
\draw(12,-8.05)--(13,-8.05);
\draw(12,-7.95)--(13,-7.95);
\node[draw,shape=circle,fill=black,scale=0.5][label=below: {$0$}] at (11,-8) {};
\node[draw,shape=rectangle,fill=white,scale=1.2] at (12,-8) {};
\node[draw,shape=circle,fill=black,scale=0.5][label=below: {$1$}] at (12,-8) {};
\node[draw,shape=circle,fill=black,scale=0.5][label=below: {$2$}] at (13,-8) {};
\draw(12.4,-7.85)--(12.6,-8)--(12.4,-8.15);
\node at (10,-8) {$\widetilde{G_2}$};
\end{tikzpicture}
\caption{The Dynkin diagrams for the affine Weyl groups.  In each case the affine node is labelled $0$, the Coxeter nodes are boxed, and the fully commutative nodes are unfilled (see Theorems~\ref{thm:coxeter-classification} and \ref{thm:fully-commutative-is-minuscule}).}
\label{fig:main-fig}
\end{figure}

\begin{table}[ht!]
\centering
\begin{tabular}{|c|c|c|}
\hline 
Type & Coweight & Reduced Word   \\\hline
$\widetilde{B_n}$ & $\omega_k^{\vee}$, $1{\leq} k{\leq} n{-}1$ & $w(s_0s_1s_2\cdots s_{n-1}s_ns_{n-1}\cdots s_{k+1})^\infty$  \\
& $\omega_{n}^{\vee}$ & $w(s_0s_2s_3\cdots s_ns_1s_2\cdots s_n)^\infty$  \\\hline
$\widetilde{C_n}$ & $\omega_k^{\vee}$, $1{\leq} k{\leq} n$ & $w(s_0s_1s_2\cdots s_ns_{n-1}\cdots s_k)^\infty$ \\\hline
$\widetilde{D_n}$ & $\omega_k^{\vee}$, $1{\leq} k{\leq} n{-}2$ & $w(s_0s_1\cdots s_{n}s_{n-2}s_{n-3}\cdots s_{k+1})^\infty$ \\
& $\omega_{n-1}^{\vee}$ & $w(s_0s_2s_3\cdots s_{n-2}s_ns_1s_2\cdots s_{n-2}s_{n-1})^\infty$  \\
& $\omega_{n}^{\vee}$ & $w(s_0s_2s_3\cdots s_{n-1}s_1s_2\cdots s_{n-2}s_n)^\infty$ 
\\\hline
$\widetilde{G_2}$ & $\omega_1^{\vee}$& $w(s_0s_1s_2)^\infty$ \\
& $\omega_{2}^{\vee}$ & $w(s_0s_1s_2s_1s_2)^\infty$   
\\\hline
$\widetilde{F_4}$ & $\omega_1^{\vee}$ & $w(s_0s_1s_2s_3s_4s_2s_3s_2)^\infty$  \\
& $\omega_{2}^{\vee}$ & $w(s_0s_1s_2s_3s_4)^\infty$  \\
& $\omega_{3}^{\vee}$ & $w(s_0s_1s_2s_3s_4s_2s_3)^\infty$  \\
& $\omega_{4}^{\vee}$ & $w(s_0s_1s_2s_3s_4s_1s_2s_3)^\infty$ 
\\\hline
$\widetilde{E_6}$ & $\omega_1^{\vee}$ & $w(s_0s_6s_3s_4s_5s_2s_3s_4s_6s_3s_2s_1)^\infty$  \\
& $\omega_{2}^{\vee}$ &  $w(s_0s_6s_3s_4s_5s_2s_3s_4s_1)^\infty$  \\
& $\omega_{3}^{\vee}$ & $w(s_0s_1s_2s_3s_4s_5s_6)^\infty$ \\
& $\omega_{4}^{\vee}$ & $w(s_0s_6s_3s_2s_1s_4s_3s_2s_5)^\infty$  \\
& $\omega_{5}^{\vee}$ & $w(s_0s_6s_3s_2s_1s_4s_3s_2s_6s_3s_4s_5)^\infty$ \\
& $\omega_{6}^{\vee}$ & $w(s_0s_6s_3s_4s_5s_2s_3s_4s_1s_2s_3)^\infty$ 
\\\hline
$\widetilde{E_7}$ & $\omega_1^{\vee}$ & $w(s_0s_3s_1s_4s_2s_7s_3s_4s_2s_7s_5s_1s_3s_2s_6s_4s_5)^\infty$  \\
& $\omega_{2}^{\vee}$ &  $w(s_0s_2s_4s_3s_2s_6s_4s_7s_1s_5s_3)^\infty$   \\
& $\omega_{3}^{\vee}$ & $w(s_0s_1s_2s_3s_4s_5s_6s_7)^\infty$ \\
& $\omega_{4}^{\vee}$ & $w(s_0s_7s_3s_5s_2s_3s_4s_6s_1s_3)^\infty$  \\
& $\omega_{5}^{\vee}$ & $w(s_0s_2s_5s_3s_4s_7s_3s_2s_1s_6s_7s_3s_4)^\infty$  \\
& $\omega_{6}^{\vee}$ & $w(s_0s_2s_1s_7s_3s_4s_2s_5s_6s_3s_4s_5s_7s_3s_4s_2s_3s_1)^\infty$  \\
& $\omega_{7}^{\vee}$ & $w(s_0s_2s_6s_7s_1s_3s_2s_3s_5s_7s_4s_5s_3s_4)^\infty$ 
\\\hline
$\widetilde{E_8}$ & $\omega_1^{\vee}$ & $w(s_0s_8s_5s_4s_6s_1s_5s_2s_3s_7s_4s_8s_6s_5s_4s_8s_3s_6s_7s_5s_2s_1s_4s_6s_3s_5s_2s_4s_3)^\infty$  \\
& $\omega_{2}^{\vee}$ &  $w(s_0s_4s_3s_1s_5s_7s_6s_4s_2s_5s_3s_6s_1s_4s_8s_2s_5s_8s_6)^\infty$  \\
& $\omega_{3}^{\vee}$ & $w(s_0s_8s_5s_4s_1s_3s_4s_7s_6s_7s_5s_2s_4s_6)^\infty$  \\
& $\omega_{4}^{\vee}$ & $w(s_0s_5s_3s_2s_1s_4s_7s_8s_5s_4s_6)^\infty$  \\
& $\omega_{5}^{\vee}$ & $w(s_0s_1s_2s_3s_4s_5s_6s_7s_8)^\infty$  \\
& $\omega_{6}^{\vee}$ & $w(s_0s_2s_1s_6s_3s_7s_4s_3s_5s_8s_6s_4s_5)^\infty$ \\
& $\omega_{7}^{\vee}$ & $w(s_0s_4s_5s_8s_6s_3s_4s_5s_4s_2s_3s_1s_2s_7s_4s_6s_3s_8s_5s_6s_8s_7s_1)^\infty$ \\
& $\omega_{8}^{\vee}$ & $w(s_0s_5s_3s_1s_8s_2s_6s_5s_4s_6s_3s_7s_6s_5s_6s_2s_4)^\infty$ \\
\hline
\end{tabular}
\caption{A list of minimal elements of $\widetilde{\mathcal{W}}$, associated to fundamental coweights as in Proposition~\ref{prop:minimal}.  In each case $w$ ranges over $W^{J_k}$ for the words corresponding to $\omega_k^{\vee}$.}
\label{tab:all-minimal}
\end{table}

\begin{remark}
Reduced expressions, in a different form from ours, for the minimal elements of $\tilde{\mc{W}}$ have recently been computed independently by Wang \cite{Wang} while this paper was in preparation. Wang also shows that infinite Coxeter elements are minimal in type $\tilde{C_n}$.
\end{remark}

\section{Background}
\label{sec:background}

\subsection{Coxeter groups}
\label{sec:Coxeter-background}
We refer the reader to Bj\"{o}rner--Brenti \cite{Bjorner-Brenti} for basics on Coxeter groups.  Let $W$ be a Coxeter group with simple reflections $S=\{s_1,\cdots, s_n\}$. Any element $c \in W$ which is the product of the $n$ simple reflections in some order is called a \emph{Coxeter element} (some other sources include conjugates of these elements in the definition of Coxeter elements, we do not). 

Given $w \in W$, an expression
\[
w=s_{i_1}\cdots s_{i_{\ell}}
\]
of minimal length is called a \emph{reduced word} for $w$, and in this case $\ell=\ell(w)$ is called the \emph{length} of $w$. The (right) weak order $\leq_R$ on $W$ is the partial order with cover relations $w \lessdot_R ws_i$ whenever $\ell(ws_i)=\ell(w)+1$. 

A well known theorem of Tits \cite{Tits-words} states that all reduced words for $w$ are connected via the defining relations $s_{i}s_js_i\cdots = s_js_is_j \cdots$ with $m_{ij} \in \{2,3,\ldots\}$ factors on each side (called a \emph{commutation move} if $m_{ij}=2$ and a \emph{braid move} if $m_{ij}\geq 3$).  If no reduced word for $w$ admits the application of a braid move then $w$ is called \emph{fully commutative} \cite{stembridge-fullycommutative}.  We say an infinite reduced word $\mathbf{i}=i_1i_2\cdots$ is fully commutative if all elements $w=s_{i_1}\cdots s_{i_k}$ are fully commutative for $k=1,2,\ldots$.

For $J \subseteq S$, the \emph{parabolic subgroup} $W_J$ is the subgroup of $W$ generated by $J$, viewed as a Coxeter group with simple reflections $J$.  Each left coset $wW_J$ of $W_J$ in $W$ contains a unique element $w^J$ of minimal length, and the set $\{w^J \: | \: w \in W\}$ of these minimal coset representatives is called the \emph{parabolic quotient} $W^J$.  Letting $w_J \in W_J$ be the unique element such that $w^Jw_J=w$, we have $\ell(w^J)+\ell(w_J)=\ell(w)$.  If $W^J$ is finite it contains a unique element $w_0^J$ of maximum length.

\subsection{Affine Weyl groups}
We refer the reader to Bourbaki \cite{Bourbaki} for more details on affine Weyl groups.  For the remainder of the paper, we let $\tilde{W}$ denote an affine Weyl group with associated finite Weyl group $W$.  We number the simple reflections so that $W$ has simple reflections $S=\{s_1,\ldots,s_n\}$ while $\tilde{W}$ has $\tilde{S}=\{s_0\} \sqcup S$.

We let $\Phi$ denote the finite root system associated to $W$, $\Phi^+$ denote a choice of positive roots, and $\Delta=\{\alpha_1,\ldots,\alpha_n\}$ denote the corresponding set of simple roots.  We write $\xi$ for the highest root of $\Phi^+$ and make the notational convention that $\alpha_0=-\xi$; we write $\widetilde{\Delta}=\{\alpha_0,\ldots,\alpha_n\}$.

The standard Euclidean space containing $\Phi$ is denoted $V$ and we write $\langle, \rangle$ for the inner product. The group $\tilde{W}$ acts faithfully on $V$ by affine linear transformations, and the action of $W \subset \tilde{W}$ is linear and preserves the inner product.  

The \emph{fundamental coweights} $\omega_1^{\vee},\ldots,\omega_n^{\vee}$ are determined by the formula $\langle \alpha_i, \omega_j^{\vee} \rangle=\delta_{ij}$.  For $i=1,\ldots,n$ the \emph{simple coroot} $\alpha_i^{\vee}$ is defined by $\alpha_i^{\vee}=\frac{2}{\langle \alpha_i, \alpha_i \rangle} \alpha_i$. The \emph{coroot lattice} is $Q^{\vee}=\bigoplus_{i=1}^n \mathbb{Z} \alpha_i^{\vee}$.  For $i=1,\ldots,n$ we let $k_i$ denote the smallest positive integer (necessarily finite) such that $k_i \omega_i^{\vee} \in Q^{\vee}$. For each $\lambda \in Q^{\vee}$ there is a unique element $t_{\lambda}$ in $\tilde{W}$ which acts on $V$ via translation by $\lambda$.  This realizes $\tilde{W}$ as the semidirect product $W \ltimes Q^{\vee}$ where $wt_{\lambda}w^{-1}=t_{w\lambda}$ for $w \in W$.

For $w \in \tilde{W}$ the \emph{inversion set} is defined to be
\[
\Inv(w)=\{\alpha_{i_1}, s_{i_1}\alpha_{i_2}, s_{i_1}s_{i_2}\alpha_{i_3},\cdots,s_{i_1}\cdots s_{i_{k-1}} \alpha_{i_k}\}
\]
where $w=s_{i_1}\cdots s_{i_k}$ is any reduced word for $w$ (it is an important fact that the inversion set does not depend on the reduced word chosen).  It is clear from the definition that if $w \leq_R w'$ then $\Inv(w) \subseteq \Inv(w')$; in fact, the converse holds as well: weak order is equivalent to containment of inversion sets. If $\mathbf{i}=i_1 i_2 \cdots$ is an infinite reduced word, then the prefixes $w^{(k)}=s_{i_1}\cdots s_{i_k} \in \tilde{W}$ clearly satisfy $w^{(k)} \leq_R w^{(k')}$ whenever $k \leq k'$.  The inversion set of $\mathbf{i}$ is defined to be the increasing union 
\[
\Inv(\mathbf{i})=\bigcup_{k=1}^{\infty} \Inv(w^{(k)}).
\]
The limit weak order $\tilde{\mc{W}}$ on the limit elements $[\mathbf{i}]$ is determined by containment of these inversion sets.

Associated to $\tilde{W}$ is an affine hyperplane arrangement $\mc{H}$ in $V$, with hyperplanes $H_{\alpha,k}=\{x \in V \: | \: \langle x, \alpha \rangle =k\}$ for $\alpha \in \Phi^+$ and $k \in \mathbb{Z}$.  The \emph{reflections} in $\tilde{W}$ are defined to be the conjugates of the simple reflections, and these elements act on $V$ via reflection about one of the hyperplanes $H_{\alpha,k}$. The connected components of the complement of $\mc{H}$ are called \emph{alcoves}, and $\tilde{W}$ acts simply transitively on the set of alcoves.  Fixing the \emph{fundamental alcove} $\mc{A}_{\mathrm{id}}$ to be that bounded by $H_{\alpha_i,0}$ for $i=1,\ldots,n$ and $H_{\alpha_0,-1}$, this action determines a canonical labelling of the alcoves $\mc{A}_w$ by elements $w \in \tilde{W}$.  

The inversions of $w$ are in natural bijection with the hyperplanes from $\mc{H}$ separating $\mc{A}_w$ from $\mc{A}_{\mathrm{id}}$: for each inversion $s_{i_1}\cdots s_{i_{a-1}}\alpha_{i_a}$ there is a reflection
\[
s_{i_1}\cdots s_{i_{a-1}}s_{i_a} s_{i_{a-1}} \cdots s_{i_1}
\]
which acts via reflection across some hyperplane $H_{\alpha,k}$, and this is the corresponding separating hyperplane. When it is convenient to argue in terms of hyperplanes, we will call this a \emph{hyperplane inversion}; since this correspondence is a bijection, weak order and limit weak order are also characterized by containment of hyperplane inversion sets.

The Dynkin diagram of $(\widetilde{W},\widetilde{S})$ is a directed graph with nodes $\widetilde{\Delta}$ such that there are $-2\langle\alpha_i,\alpha_j\rangle/\langle\alpha_j,\alpha_j\rangle$ directed edges from $\alpha_i$ to $\alpha_j$, for $i\neq j$. Finite and affine Weyl groups are completely classified by their Dynkin diagrams.  They consist of four infinite families $\tilde{A_n} (n\geq 1), \tilde{B_n} (n \geq 2), \tilde{C_n} (n \geq 2), \tilde{D_n} (n \geq 4)$ and the exceptional types $\tilde{E_6}, \tilde{E_7}, \tilde{E_8}, \tilde{F_4},$ and $\tilde{G_2}$. See the corresponding Dynkin diagrams in Figure~\ref{fig:main-fig}.

\section{Words for infinite translation elements}
\label{sec:translation}

\subsection{Translations by fundamental coweights}
The following proposition, implicit in \cite{lam-pylyavskyy} and \cite{lam-thomas}, describes the minimal elements of $\tilde{\mc{W}}$ geometrically: they are the infinite translations in the directions of the rays of the corresponding reflection arrangement.

\begin{prop}\label{prop:minimal}
The minimal elements of $\widetilde{\mathcal{W}}$ are precisely $$\{[t_{wk_i\omega_i^{\vee}}^{\infty}]\:|\: 1\leq i\leq n, w\in W^{J_i}\}$$
where $J_i \coloneqq \{s_j \in S \:|\: j\neq i\}$.
\end{prop}

In Section~\ref{sec:explicit-words} we give a method for constructing infinite reduced words for these and other infinite translation elements.  Understanding these reduced words is necessary for resolving Problems~\ref{prob:describe-words} and \ref{prob:are-minimal-fully-commutative} and understanding the limit Coxeter elements, for the characterization of minimal elements in Proposition~\ref{prop:minimal} is not immediately applicable to any of these problems.

\subsection{Explicit reduced words}
\label{sec:explicit-words}
In this section, we explain how to write down explicit infinite reduced words that correspond to open faces of the reflection arrangement of $W$. The content of this section generalizes that of Section 4.7 of \cite{lam-pylyavskyy}, which is specific to type $A$. Our formulation and arguments are type-uniform and the proof ideas will be different from those in \cite{lam-pylyavskyy}.

Recall that the set of simple roots for $W$ is $\Delta=\{\alpha_1,\ldots,\alpha_n\}$ while the set of simple roots for $\widetilde{W}$ is $\widetilde{\Delta}=\{\alpha_0,\ldots,\alpha_n\}$, where $\alpha_0=-\xi$ with $\xi$ the highest root of $W$. 

Let $\lambda\neq0\in Q^{\vee}$. We now explicitly write down an infinite reduced word $\mathbf{i}=s_{i_1}s_{i_2}\cdots$ such that $[\mathbf{i}]=[t_{\lambda}^\infty]$. The construction is inductive. Let $\lambda^{(0)}=\lambda$. For $j\geq1$, we choose $i_j\in\{0,1,\ldots,n\}$ such that $\langle \lambda^{(j-1)},\alpha_{i_j}\rangle <0$ and then let $\lambda^{(j)}=s_{i_j}\lambda^{(j-1)}.$

Notice that if $\langle\lambda,\alpha_k\rangle\geq0$ for all $k=0,\ldots,n$, then we must have $\langle\lambda,\alpha_k\rangle=0$ for all $k=0,\ldots,n$ since $-\alpha_0=\xi$ is a positive linear combination of $\alpha_1,\ldots,\alpha_n$. And since $\alpha_1,\ldots,\alpha_n$ span $V$, the equalities imply $\lambda=0$. Therefore, as long as $\lambda\neq0$, none of its Weyl group translates will be 0 so the above procedure will continue indefinitely. 

\begin{prop}\label{prop:explicitword}
Let $\lambda\neq0\in Q^{\vee}$ and construct the infinite word $\mathbf{i}$ as above. Then $\mathbf{i}$ is reduced and $[\mathbf{i}]=[t_{\lambda}^\infty]$.
\end{prop}
\begin{proof}
Recall that the fundamental alcove of the affine hyperplane arrangement of $\widetilde{W}$ is given by
$$\mc{A}_{\mathrm{id}}=\{x\in V\:|\:\langle x,\alpha_i\rangle>0\text{ for }i=1,\ldots,n\text{ and }\langle x,\alpha_0\rangle>-1\}.$$
For any $\alpha_i$ such that $\langle \lambda,\alpha_i\rangle<0$, we can choose $\mu\in\mc{A}_{\mathrm{id}}$ close to the hyperplane $H_{\alpha_i,k}$ ($k=0$ if $i=1,\ldots,n$ and $k=-1$ if $i=0$ where we identify $H_{\alpha_0,-1}$ with $H_{\xi,1}$) that bounds $\mc{A}_{\mathrm{id}}$ such that moving in direction $\lambda$ from $\mu$ intersects $H_{\alpha_i,k}$ first, among all $n+1$ hyperplanes that bound $\mc{A}_{\mathrm{id}}$. Conversely, if $\langle\lambda,\alpha_i\rangle\geq0$, then from any $\mu\in\mc{A}_{\mathrm{id}}$ and moving in direction $\lambda$, we are only getting further away from $H_{\alpha_i,k}$ and can never encounter this hyperplane.

After such an $\alpha_{i_1}$ is chosen with $\langle\lambda,\alpha_{i_1}\rangle<0$, we move into the alcove $\mc{A}_{s_{i_1}}$, which is the alcove reflected across the hyperplane $H_{\alpha_{i_1},k}$ from the fundamental alcove, and is also the alcove that can be reached from some point in $\mc{A}_{\mathrm{id}}$ by moving in direction $\lambda$. Reflecting by $s_{i_1}$ and choosing $\alpha_{i_2}$ such that $\langle s_{i_1}\lambda, \alpha_{i_2}\rangle<0$ is exactly the same as choosing an alcove $\mc{A}_{s_{i_2}s_{i_1}}$ which can be reached from some point $\mu\in \mc{A}_{s_{i_1}}$ by translating in direction $\lambda$.

Continue this procedure described above, the alcove path described by $\mathbf{i}$ is a sequence of alcoves starting at $\mc{A}_{\mathrm{id}}$ such that the next one can be obtained by moving from some point inside the previous alcove in direction $\lambda$. As a result, we see that no hyperplanes can be crossed twice by this alcove path $\mathbf{i}$ and that the hyperplanes crossed are exactly those crossed by $t_{\lambda}^\infty$. Therefore, $\mathbf{i}$ is reduced and $[\mathbf{i}]=[t_{\lambda}^\infty]$.
\end{proof}

\begin{remark}
When for some $k$ we have $\lambda^{(k)}=\lambda$ in the above procedure, we can conclude immediately that $[t_{\lambda}^\infty]=[(s_{i_1}\cdots s_{i_k})^\infty]$. This is the method used to compute the words appearing in Table~\ref{tab:all-minimal}.
\end{remark}

\section{Limit Coxeter elements}
\label{sec:infinite-coxeter}

\begin{prop}
\label{prop:coxeters-are-conjugate}
Let $\tilde{W}$ be an affine Weyl group other than the affine symmetric group, and let $c, c'$ be any two Coxeter elements for $\tilde{W}$, then $c$ and $c'$ are $W$-conjugate.
\end{prop}
\begin{proof}
It is well-known (and easy to verify) that the distinct Coxeter elements $c$ for any Coxeter group correspond naturally to the acyclic orientations $\mc{O}$ of the edges of the Dynkin diagram, with a directed edge from $\alpha_i$ to an adjacent node $\alpha_j$ indicating that $s_i$ precedes $s_j$ in the product defining $c$.  

Let $c,c'$ be two Coxeter elements of $\tilde{W}$ with corresponding orientations $\mc{O}, \mc{O}'$.  Conjugating $c$ by $s_i$ corresponds to reversing the orientation of all edges incident to the node $\alpha_i$ when this node is a source or a sink.  Since the Dynkin diagram is a tree, it is not hard to see that we may move from $\mc{O}$ to $\mc{O}'$ by a sequence of such moves, so $c,c'$ are $\tilde{W}$-conjugate.  To see that they are in fact $W$-conjugate, note that reversing orientations at the single node $\alpha_0$ has the same effect as reversing at every node except $\alpha_0$.  Therefore we can connect $c$ and $c'$ without ever conjugating by $s_0$.
\end{proof}

\begin{cor}
\label{cor:coxeter-node-makes-sense}
If $\tilde{W}$ is an affine Weyl group other than the affine symmetric group and if we have $[c^{\infty}]=[t_{wk_i\omega_i^{\vee}}^{\infty}]$ for some Coxeter element $c$ for $\tilde{W}$ and some $w \in W$, then for every Coxeter element $c'$ we have
\[
[(c')^{\infty}]=[t_{uk_i\omega_i^{\vee}}^{\infty}]
\]
for some $u \in W$.
\end{cor}
\begin{proof}
We must have $c^a=t_{wk_i\omega_i^{\vee}}^b$ for some positive integers $a,b$.  Let $v \in W$ be such that $vcv^{-1}=c'$ (guaranteed to exist by Proposition~\ref{prop:coxeters-are-conjugate}), then we have
\begin{align*}
    (c')^a&=(vcv^{-1})^a \\
    &=vc^av^{-1} \\
    &=vt_{wk_i\omega_i^{\vee}}^bv^{-1} \\
    &=vt_{wbk_i\omega_i^{\vee}}v^{-1} \\
    &=t_{vwk_i\omega_i^{\vee}}^b.
\end{align*}
Thus we can take $u=vw$.
\end{proof}

In light of Corollary~\ref{cor:coxeter-node-makes-sense}, we say $\alpha_i$ is a \emph{Coxeter node} for $\tilde{W}$ if $[t_{wk_i\omega_i^{\vee}}^{\infty}]=[c^{\infty}]$ for some $w \in W$ and some Coxeter element $c$.  By Corollary~\ref{cor:coxeter-node-makes-sense} (except when $\tilde{W}$ is the affine symmetric group, where all nodes are Coxeter nodes by Theorem~\ref{thm:type-A}) the Coxeter node is unique if it exists.

In the classification of irreducible finite root systems a standard reduction uses the fact that (except in type $A$) every irreducible Dynkin diagram contains either a unique node $\alpha_i$ adjacent to three other nodes or a unique multiple edge.  This multiple edge, if it exists, connects two nodes whose corresponding simple roots have different lengths; call the longer one $\alpha_i$.  In either case, we say $\alpha_i$ is the \emph{heavy node}.

Recall that $W$ is a finite Weyl group so its Dynkin diagram is a tree, which is a bipartite graph. Let $S=F\sqcup U$ (where $F$ stands for ``filled" and $U$ stands for ``unfilled") be a bipartition of the Dynkin diagram, and define two elements $s_F:=\prod_{s_i\in F}s_i$ and $s_U:=\prod_{s_i\in U}s_i$ in the Weyl group. Notice that since both $F$ and $U$ are disconnected, the order in which we take these products does not matter. Also recall that $\xi$ is the highest root of the root system $\Phi$, and $s_i\xi=\xi$ for all but one $s_i\in S$ when $\Phi$ is not of type $A$. This means that exactly one of $s_F$ and $s_U$ fixes $\xi$.

The following theorem is the main result of this section, which provides various characterizations for the Coxeter node.
\begin{theorem}
\label{thm:coxeter-classification}
Let $\tilde{W}$ be an affine Weyl group other than the affine symmetric group.  The following are equivalent for a node $\alpha_i$ in the Dynkin diagram:
\begin{enumerate}
    \item $\alpha_i$ is a Coxeter node;
    \item $\alpha_i$ is the heavy node;
    \item the action $\cdots s_Us_Fs_Us_F\xi$ of some finite alternating product of $s_U,s_F$ on $\xi$ is equal to the root $\alpha_i$. 
\end{enumerate}
\end{theorem}

\begin{cor}
A Coxeter node exists in all affine types (and is unique except for the affine symmetric group) so infinite powers $[c^\infty]$ of Coxeter elements are always minimal in $\tilde{\mc{W}}$.
\end{cor}

\begin{remark}
Note that since one of $s_U$, $s_F$ fixes $\xi$, whether $s_U$ or $s_F$ is first applied to $\xi$ does not matter. In addition, both the uniqueness and existence of the simple root $\alpha_i$ satisfying (3) in Theorem~\ref{thm:coxeter-classification} should not be clear from the definition.
\end{remark}

\begin{ex}\label{ex:bipartite-E6}
We provide an explicit calculation for condition (3) of Theorem~\ref{thm:coxeter-classification} in type $E_6$. The highest root is $\xi=\alpha_1+2\alpha_2+3\alpha_3+2\alpha_4+\alpha_5+2\alpha_6$ (see the root indexing in Figure~\ref{fig:main-fig}). Consider $F=\{\alpha_1,\alpha_3,\alpha_5\}$ and $U=\{\alpha_2,\alpha_4,\alpha_6\}$ as shown in Figure~\ref{fig:bipartite-E6}, where we write down the coefficients for each of the roots $\xi,s_U\xi,s_Fs_U\xi,\ldots$ as linear combinations of the simple roots. As we can see, in the end, $s_Us_Fs_Us_Fs_U\xi=\alpha_3$ as desired.

\begin{figure}[ht!]
\centering
\begin{tikzpicture}[scale=0.7]
\draw(0,0)--(4,0);
\draw(2,0)--(2,1);
\node[draw,shape=circle,fill=black,scale=0.5][label=below: {$1$}] at (0,0) {};
\node[draw,shape=circle,fill=white,scale=0.5][label=below: {$2$}] at (1,0) {};
\node[draw,shape=circle,fill=black,scale=0.5][label=below: {$3$}] at (2,0) {};
\node[draw,shape=circle,fill=white,scale=0.5][label=below: {2}] at (3,0) {};
\node[draw,shape=circle,fill=black,scale=0.5][label=below: {1}] at (4,0) {};
\node[draw,shape=circle,fill=white,scale=0.5][label=right: {2}] at (2,1) {};
\node at (-1,0) {$\xi$};

\draw(8,0)--(12,0);
\draw(10,0)--(10,1);
\node[draw,shape=circle,fill=black,scale=0.5][label=below: {$1$}] at (8,0) {};
\node[draw,shape=circle,fill=white,scale=0.5][label=below: {$2$}] at (9,0) {};
\node[draw,shape=circle,fill=black,scale=0.5][label=below: {$3$}] at (10,0) {};
\node[draw,shape=circle,fill=white,scale=0.5][label=below: {2}] at (11,0) {};
\node[draw,shape=circle,fill=black,scale=0.5][label=below: {1}] at (12,0) {};
\node[draw,shape=circle,fill=white,scale=0.5][label=right: {1}] at (10,1) {};
\node at (7,0) {$s_U\xi$};

\draw(0,-2.5)--(4,-2.5);
\draw(2,-2.5)--(2,-1.5);
\node[draw,shape=circle,fill=black,scale=0.5][label=below: {$1$}] at (0,-2.5) {};
\node[draw,shape=circle,fill=white,scale=0.5][label=below: {$2$}] at (1,-2.5) {};
\node[draw,shape=circle,fill=black,scale=0.5][label=below: {$2$}] at (2,-2.5) {};
\node[draw,shape=circle,fill=white,scale=0.5][label=below: {2}] at (3,-2.5) {};
\node[draw,shape=circle,fill=black,scale=0.5][label=below: {1}] at (4,-2.5) {};
\node[draw,shape=circle,fill=white,scale=0.5][label=right: {1}] at (2,-1.5) {};
\node at (-1.5,-2.5) {$s_Fs_U\xi$};

\draw(8,-2.5)--(12,-2.5);
\draw(10,-2.5)--(10,-1.5);
\node[draw,shape=circle,fill=black,scale=0.5][label=below: {$1$}] at (8,-2.5) {};
\node[draw,shape=circle,fill=white,scale=0.5][label=below: {$1$}] at (9,-2.5) {};
\node[draw,shape=circle,fill=black,scale=0.5][label=below: {$2$}] at (10,-2.5) {};
\node[draw,shape=circle,fill=white,scale=0.5][label=below: {1}] at (11,-2.5) {};
\node[draw,shape=circle,fill=black,scale=0.5][label=below: {1}] at (12,-2.5) {};
\node[draw,shape=circle,fill=white,scale=0.5][label=right: {1}] at (10,-1.5) {};
\node at (6.5,-2.5) {$s_Us_Fs_U\xi$};

\draw(0,-5)--(4,-5);
\draw(2,-5)--(2,-4);
\node[draw,shape=circle,fill=black,scale=0.5][label=below: {$0$}] at (0,-5) {};
\node[draw,shape=circle,fill=white,scale=0.5][label=below: {$1$}] at (1,-5) {};
\node[draw,shape=circle,fill=black,scale=0.5][label=below: {$1$}] at (2,-5) {};
\node[draw,shape=circle,fill=white,scale=0.5][label=below: {1}] at (3,-5) {};
\node[draw,shape=circle,fill=black,scale=0.5][label=below: {0}] at (4,-5) {};
\node[draw,shape=circle,fill=white,scale=0.5][label=right: {1}] at (2,-4) {};
\node at (-1.5,-5) {$(s_Fs_U)^2\xi$};

\draw(8,-5)--(12,-5);
\draw(10,-5)--(10,-4);
\node[draw,shape=circle,fill=black,scale=0.5][label=below: {$0$}] at (8,-5) {};
\node[draw,shape=circle,fill=white,scale=0.5][label=below: {$0$}] at (9,-5) {};
\node[draw,shape=circle,fill=black,scale=0.5][label=below: {$1$}] at (10,-5) {};
\node[draw,shape=circle,fill=white,scale=0.5][label=below: {0}] at (11,-5) {};
\node[draw,shape=circle,fill=black,scale=0.5][label=below: {0}] at (12,-5) {};
\node[draw,shape=circle,fill=white,scale=0.5][label=right: {0}] at (10,-4) {};
\node at (6.3,-5) {$s_U(s_Fs_U)^2\xi$};
\end{tikzpicture}
\caption{Example~\ref{ex:bipartite-E6} of (3) in Theorem~\ref{thm:coxeter-classification} for type $E_6$: $\alpha_3=s_Us_Fs_Us_Fs_U\xi.$}
\label{fig:bipartite-E6}
\end{figure}
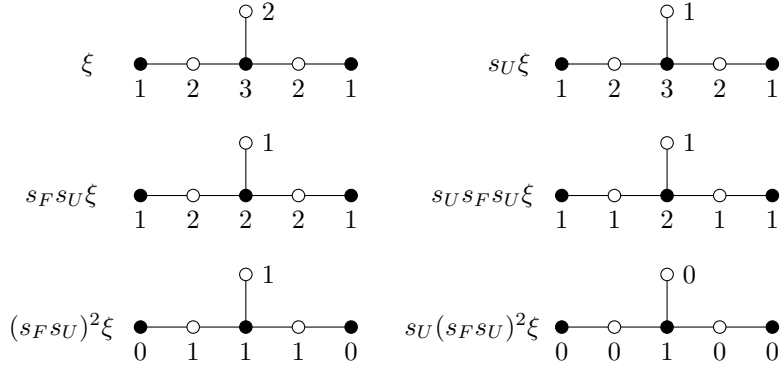
\end{ex}

\begin{proof}[Proof of Theorem~\ref{thm:coxeter-classification}]
We first establish the equivalence between (2) and (3). To prove (2)$\Rightarrow$(3), it suffices to provide an explicit realization of the heavy node $\alpha_i$ as $\cdots s_Us_F\xi$ for each type; this computation is straightforward. Let the heavy node $\alpha_i\in F$ be filled (see Figure~\ref{fig:main-fig} for the labels), so that the bipartite coloring is determined. Table~\ref{tab:bipartite} provides the necessary realizations and establishes the existence for (3). 
\begin{table}[ht!]
\centering
\begin{tabular}{c|c}
Type & Heavy node \\\hline
$B_n$ & $\alpha_{n-1}=(s_Us_F)^{\left\lfloor n/2\right\rfloor}\xi$ \\
$C_n$ & $\alpha_{n}=(s_Us_F)^{\left\lfloor n/2\right\rfloor}\xi$ \\
$D_n$ & $\alpha_{n-2}=(s_Us_F)^{\left\lfloor (n-1)/2\right\rfloor}\xi$ \\
$E_6$ & $\alpha_3=(s_Us_F)^3\xi$ \\
$E_7$ & $\alpha_3=(s_Us_F)^4\xi$ \\
$E_8$ & $\alpha_5=(s_Us_F)^7\xi$ \\
$F_4$ & $\alpha_2=(s_Us_F)^3\xi$ \\
$G_2$ & $\alpha_2=s_Us_F\xi$
\end{tabular}
\caption{The heavy node as $\alpha_i=\cdots s_Us_Fs_Us_F\xi$ for all finite Weyl groups not of type $A$.}
\label{tab:bipartite}
\end{table}

It is easy to also check that in all cases, when we have $(s_Us_F)^k\xi=\alpha_i$ the heavy node, the roots $\xi,s_F\xi$ (which might be equal to $\xi$), $s_Us_F\xi,\ldots,(s_Us_F)^k\xi$ are all positive with decreasing height and contain exactly one simple root, the heavy node, in their common support. Then $s_U\alpha_i=-\alpha_i$ so going back up in height, the orbit of $\xi$ under the subgroup of $W$ generated by $s_U,s_F$ consists of $\pm\xi,\pm s_F\xi,\ldots,\pm\alpha_i$, with only one simple root. This establishes the uniqueness for (3), and thus we have that (2) is equivalent to (3).

We now show that (2) and (3) imply (1), which is the key of this theorem. If this is proved, we establish the existence of the Coxeter node. And since we know that Coxeter is unique if it exists by Corollary~\ref{cor:coxeter-node-makes-sense}, we can conclude that the Coxeter node has to be the heavy node, which means that the direction (1)$\Rightarrow$(2) will be established as well. 

For this part of the paper only, we will use a different labeling from Figure~\ref{fig:main-fig}. Since the Dynkin diagram is a tree, we make can make it into a poset by setting $u \leq v$ if $v$ is on the unique simple path between $u$ and the heavy node. The heavy node is then the unique maximum element of this poset and we label the simple roots by taking any arbitrary linear extension, thus $\alpha_n$ will be the heavy node. An example is shown in Figure~\ref{fig:new-label-E8}. For $i=1,\ldots,n-1$, define $g(i)=j$ if $\alpha_j$ is adjacent to $\alpha_i$ in the Dynkin diagram and on the path from $\alpha_i$ to $\alpha_n$. For example in Figure~\ref{fig:new-label-E8}, $g(1)=2$, $g(2)=3$, $g(7)=8$, $g(5)=6$, etc. 
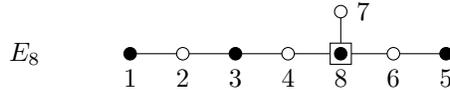
\begin{figure}[ht!]
\centering
\begin{tikzpicture}[scale=0.7]
\draw(12,-4)--(18,-4);
\draw(16,-4)--(16,-3.2);
\node[draw,shape=circle,fill=black,scale=0.5][label=below: {$1$}] at (12,-4) {};
\node[draw,shape=circle,fill=white,scale=0.5][label=below: {$2$}] at (13,-4) {};
\node[draw,shape=circle,fill=black,scale=0.5][label=below: {$3$}] at (14,-4) {};
\node[draw,shape=circle,fill=white,scale=0.5][label=below: {$4$}] at (15,-4) {};
\node[draw,shape=rectangle,fill=white,scale=1.2] at (16,-4) {};
\node[draw,shape=circle,fill=black,scale=0.5][label=below: {$8$}] at (16,-4) {};
\node[draw,shape=circle,fill=white,scale=0.5][label=below: {$6$}] at (17,-4) {};
\node[draw,shape=circle,fill=black,scale=0.5][label=below: {$5$}] at (18,-4) {};
\node[draw,shape=circle,fill=white,scale=0.5][label=right: {$7$}] at (16,-3.2) {};
\node at (10,-4) {$E_8$};
\end{tikzpicture}
\caption{A new labeling for the Dynkin diagram of type $E_8$ used in the proof of Theorem~\ref{thm:coxeter-classification}.}
\label{fig:new-label-E8}
\end{figure}

Write $b_i=\langle\xi,\omega_i^{\vee}\rangle$ for all $i$. In other words, $\xi=\sum_{i=1}^n b_i\alpha_i$. For $j=1,\ldots,n$, write $\xi^{(j)}=\sum_{i\geq j}b_i\alpha_i$. For example, $\xi^{(1)}=\xi$ and $\xi^{(n)}=b_n\alpha_n$; typically these are not roots. Define a weight $\lambda\in Q$ by
\[
\lambda:=\sum_{i=1}^n \big(\langle \alpha_i^{\vee},\xi^{(i)}\rangle-b_i\big)\omega_i.
\]
For the remainder of the proof, we will first use Proposition~\ref{prop:explicitword} to show that $[t_{\lambda}^{\infty}]=[c^{\infty}]$ for some Coxeter element $c\in\widetilde{W}$, and then use condition (3) to show that $\lambda$ and $\omega_n$ are in the same $W$-orbit.

\

\noindent\textbf{Step one:} prove that $[t_{\lambda}^{\infty}]=[c^{\infty}]$ for $c=s_0s_1\cdots s_n$ with the linear extension labelling.

We start with $\lambda^{(0)}=\lambda$ and go through the procedure described in Section~\ref{sec:explicit-words}. By definition of our labels, for $i<n$:
\[
\langle \alpha_i^{\vee},\xi^{(i)}\rangle=\langle\alpha_i^{\vee},b_i\alpha_i+b_{g(i)}\alpha_{g(i)}\rangle=2b_i+b_{g(i)}\langle\alpha_i^{\vee},\alpha_{g(i)}\rangle.
\]
Continuing the calculation:
\begin{align*}
\langle\lambda,\xi\rangle=&\sum_{i=1}^{n-1}\big(b_i+b_{g(i)}\langle\alpha_i^{\vee},\alpha_{g(i)}\rangle\big)b_i\langle\alpha_i,\omega_i\rangle+ b_n^2\langle \alpha_n,\omega_n\rangle\\
=&\sum_{i=1}^nb_i^2\langle\alpha_i,\omega_i\rangle+\sum_{i=1}^{n-1}b_ib_{g(i)}\langle \alpha_i^{\vee},\alpha_{g(i)}\rangle\langle\alpha_i,\omega_i\rangle\\
=&\sum_{i=1}^n\frac{\langle\alpha_i,\alpha_i\rangle}{2}b_i^2+\sum_{i=1}^{n-1}b_ib_{g(i)}\langle\alpha_i,\alpha_{g(i)}\rangle\\
=&\frac{1}{2}\left\langle \sum_{i=1}^n b_i\alpha_i,\sum_{i=1}^nb_i\alpha_i\right\rangle=\frac{1}{2}\langle\xi,\xi\rangle
\end{align*}
where the second to last equality comes from the fact that as $i$ ranges from $1$ to $n-1$, the pairs $(i,g(i))$ range over all edges in the Dynkin diagram.

This calculation shows that $\langle\lambda,\xi\rangle>0$ so $\langle\lambda,\alpha_0\rangle<0$ where $\alpha_0=-\xi$. By Proposition~\ref{prop:explicitword}, to obtain an infinite reduced word $\mathbf{j}=j_1j_2\ldots$ such that $[\mathbf{j}]=[t_{\lambda}^\infty]$, we can choose our first index $j_1$ to be $0$ and 
\[
\lambda^{(1)}=s_0\lambda^{(0)}=\lambda-\frac{2\langle\lambda,\xi\rangle}{\langle\xi,\xi\rangle}\xi=\lambda-\xi.
\]
We now show via induction that we can choose $j_i=i-1$ for $i=1,2,\ldots,n+1$, and obtain $\lambda^{(i)}=\lambda-\xi^{(i)}$. The base case $i=1$ is already done. Suppose that we have $\lambda^{(i)}=\lambda-\xi^{(i)}$, and consider
\begin{align*}
\langle\lambda^{(i)},\alpha_i\rangle=&\langle \lambda-\xi^{(i)},\alpha_i\rangle\\
=&\langle\lambda,\alpha_i\rangle-\langle\xi^{(i)},\alpha_i\rangle\\
=&(\langle\alpha_i^{\vee},\xi^{(i)}\rangle-b_i)\langle\alpha_i,\omega_i\rangle-\langle\xi^{(i)},\alpha_i\rangle\\
=&-\frac{\langle\alpha_i,\alpha_i\rangle}{2}b_i<0.
\end{align*}
Thus, we can choose $j_{i+1}=i$ and 
\[
\lambda^{(i+1)}=s_i\lambda^{(i)}=\lambda^{(i)}-\frac{2\langle\lambda^{(i)},\alpha_i\rangle}{\langle\alpha_i,\alpha_i\rangle}\alpha_i=\lambda-\xi^{(i)}-b_i\alpha_i=\lambda-\xi^{(i+1)}
\]
as desired. Thus, we arrive at $\lambda^{(n+1)}=\lambda-\xi^{(n+1)}=\lambda$. As we have returned to $\lambda$, continuing this procedure gives $[t_{\lambda}^{\infty}]=[c^{\infty}]$ for $c=s_0s_1\cdots s_n$.

\

\noindent\textbf{Step two:} prove that $\lambda=w\mu$ for some $w\in W$ and some intermediate weight $\mu$ which resembles the highest root $\xi$ in certain ways.

For $k=0,1,\ldots,m$, let $D_k$ be the set of nodes in the Dynkin diagram of the finite root system with distance exactly $k$ to the heavy node $\alpha_n$ where $m$ is the maximum distance from $\alpha_n$. Let $N_k=D_k\sqcup D_{k-2}\sqcup D_{k-4}\cdots$. For each pair of nodes $i,j$, we also let $d(i,j)$ denote the distance in the Dynkin diagram between $\alpha_i$ and $\alpha_j$. In particular, $N_0=\{\alpha_n\}$ and each $N_k$ consists of disconnected set of nodes since the Dynkin diagram is a tree. For a set of disconnected nodes $N$, write $s_N=\prod_{i\in N}s_i$ which is an involution. Also recall that $F=N_0\sqcup N_2\sqcup\cdots$ and $U=N_1\sqcup N_3\sqcup\cdots$ given our chosen bipartite coloring. 

Let $\alpha_a\in\Delta$ be the unique neighbor of $\alpha_0$ in the affine Dynkin diagram. In other words, it is the unique simple root such that $\langle\alpha_a,\xi\rangle\neq0$, in which case $\langle\alpha_a,\xi\rangle>0$. We observe that $\alpha_a\in N_{m-1}\cup N_m$ for all types of interest.

Construct weights $\mu^{(0)},\mu^{(1)},\ldots,\mu^{(m)}$ such that $\mu^{(i)}=s_{N_{i-1}}\cdots s_{N_0}\lambda$ for all $i$. We have $\mu^{(0)}=\lambda$. We study these weights by expanding them in the basis of the fundamental weights $\{\omega_i\}_{i=1}^n$. Write $\mu^{(j)}=\sum_{i=1}^n c_i^{(j)}\omega_i$. We use induction on $j$ to prove the following:
\[
c_i^{(j)}=
\begin{cases}
(-1)^{j-d(i,n)}b_i, & \text{ if }d(i,n)\leq j,\\
\langle\alpha_i^{\vee},\xi^{(i)}\rangle-b_i, & \text{ if }d(i,n)>j.
\end{cases}
\]
The base case $j=0$ is clear since $\mu^{(0)}=\lambda$, where the only $i$ with $d(i,n)=0$ is $i=n$ and we know that here $\langle\alpha_n^{\vee},\xi^{(n)}\rangle-b_n=2b_n-b_n=b_n$ as desired.

Now, by considering the inner product with fundamental coweights, we can write each root in the basis $\{\omega_i\}_{i=1}^n$ as
\[
\alpha_i=2\omega_i+\sum_{j\sim i}\langle \alpha_i,\alpha_j^{\vee}\rangle\omega_j
\]
where $j\sim i$ means that $d(i,j)=1$. Thus, in the basis $\{\omega_i\}_{i=1}^n$, when reflecting by $s_i$, only the coefficients with distance at most 1 to $\alpha_i$ can be affected. This means that when $d(i,n)>j$, $c_i^{(j)}$ remains unchanged from $c_i^{(0)}=\langle\alpha_i^{\vee},\xi^{(i)}\rangle-b_i$. If we have a weight $\nu=\sum c_i\omega_i$, then
\[
s_i\nu=\nu-\langle \alpha_i^{\vee},\nu\rangle\alpha_i=\nu-c_i\left(2\omega_i+\sum_{j\sim i}\langle\alpha_i,\alpha_j^{\vee}\rangle\omega_j\right).
\]

Suppose that we have $\mu^{(j)}$ as described above. Consider $\mu^{(j+1)}=s_{N_j}\mu^{(j)}$ and some $i$ with $d(i,n)\leq j+1$. If $d(i,n)=j+1$, the only $s_k$'s in $s_{N_j}$ that affect the coefficient at $\omega_i$ is $s_{g(i)}$. In this case,
\begin{align*}
c_i^{(j+1)}=&c_i^{(j)}-c_{g(i)}^{(j)}\langle\alpha_i^{\vee},\alpha_{g(i)}\rangle\\
=&\langle \alpha_i^{\vee},\xi^{(i)}\rangle-b_i-(-1)^{j-d(g(i),n)}b_{g(i)}\langle\alpha_i^{\vee},\alpha_{g(i)}\rangle\\
=&(2b_i+b_{g(i)}\langle\alpha_i^{\vee},\alpha_{g(i)}\rangle)-b_i-b_{g(i)}\langle\alpha_i^{\vee},\alpha_{g(i)}\rangle=b_i
\end{align*}
as desired. If $d(i,n)\leq j$ and $d(i,n)\equiv j\bmod 2$, i.e. $\alpha_i\in N_j$, we have
\[
c_i{(j+1)}=c_i^{(j)}-2c_i^{(j)}=(-1)^{j+1-d(i,n)}b_i
\]
as desired. Finally, if $d(i,n)<j$ and $\alpha_i\notin N_j$, as $j\leq m-1$, $d(i,n)\leq m-2$ so $\alpha_i$ cannot possibly be $\alpha_a$, thus we have $\langle\alpha_i,\xi\rangle=0$. Expanding, we obtain
\[
0=\langle\alpha_i,\xi\rangle=b_i\langle\alpha_i,\alpha_i\rangle+\sum_{i'\sim i}b_{i'}\langle\alpha_i,\alpha_{i'}\rangle\\
\Rightarrow 0=2b_i+\sum_{i'\sim i}b_{i'}\langle\alpha_i^{\vee},\alpha_{i'}\rangle.
\]
Consequently,
\begin{align*}
c_i^{(j+1)}=&c_i^{(j)}-\sum_{i'\sim i}c_{i'}^{(j)}\langle \alpha_i^{\vee},\alpha_{i'}\rangle\\
=&(-1)^{j-d(i,n)}b_i-\sum_{i'\sim i}(-1)^{j-d(i',n)}b_{i'}\langle \alpha_i^{\vee},\alpha_{i'}\rangle\\
=&(-1)^{j+1-d(i,n)} \left(-b_i-\sum_{i'\sim i}b_{i'}\langle\alpha_i^{\vee},\alpha_{i'}\rangle\right)\\
=&(-1)^{j+1-d(i,n)}b_i.
\end{align*}
The induction step goes through. In the end, we see that $\mu:=\mu^{(m)}$ in the $\{\omega_i\}_{i=1}^n$ basis has the same coefficients as $\xi$ in the $\{\alpha_i\}_{i=1}^n$ basis, but with alternating signs. With the same calculation as above, we also see that if $a\in s_F$, then $s_F\mu=-\mu$ and if $a\in s_U$, $s_U\mu=-\mu$. Both of $\mu$ and $-\mu$ are in the $W$-orbit of $\lambda$. For the next step, we consider them together. 

An example of $\mu^{(0)},\mu^{(1)},\ldots,\mu,-\mu$ for type $E_6$ is shown in Figure~\ref{fig:mu-E6}, where every weight is labeled by its coefficients when decomposed in the basis $\{\omega_i\}_{i=1}^n$.
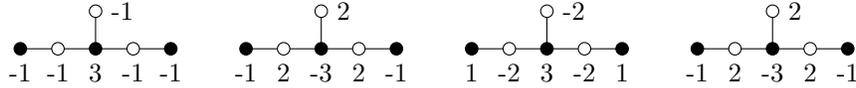
\begin{figure}[ht!]
\centering
\begin{tikzpicture}[scale=0.5]
\draw(0,0)--(4,0);
\draw(2,0)--(2,1);
\node[draw,shape=circle,fill=black,scale=0.5][label=below: {-1}] at (0,0) {};
\node[draw,shape=circle,fill=white,scale=0.5][label=below: {-1}] at (1,0) {};
\node[draw,shape=circle,fill=black,scale=0.5][label=below: {3}] at (2,0) {};
\node[draw,shape=circle,fill=white,scale=0.5][label=below: {-1}] at (3,0) {};
\node[draw,shape=circle,fill=black,scale=0.5][label=below: {-1}] at (4,0) {};
\node[draw,shape=circle,fill=white,scale=0.5][label=right: {-1}] at (2,1) {};

\draw(6,0)--(10,0);
\draw(8,0)--(8,1);
\node[draw,shape=circle,fill=black,scale=0.5][label=below: {-1}] at (6,0) {};
\node[draw,shape=circle,fill=white,scale=0.5][label=below: {2}] at (7,0) {};
\node[draw,shape=circle,fill=black,scale=0.5][label=below: {-3}] at (8,0) {};
\node[draw,shape=circle,fill=white,scale=0.5][label=below: {2}] at (9,0) {};
\node[draw,shape=circle,fill=black,scale=0.5][label=below: {-1}] at (10,0) {};
\node[draw,shape=circle,fill=white,scale=0.5][label=right: {2}] at (8,1) {};

\draw(12,0)--(16,0);
\draw(14,0)--(14,1);
\node[draw,shape=circle,fill=black,scale=0.5][label=below: {1}] at (12,0) {};
\node[draw,shape=circle,fill=white,scale=0.5][label=below: {-2}] at (13,0) {};
\node[draw,shape=circle,fill=black,scale=0.5][label=below: {3}] at (14,0) {};
\node[draw,shape=circle,fill=white,scale=0.5][label=below: {-2}] at (15,0) {};
\node[draw,shape=circle,fill=black,scale=0.5][label=below: {1}] at (16,0) {};
\node[draw,shape=circle,fill=white,scale=0.5][label=right: {-2}] at (14,1) {};

\draw(18,0)--(22,0);
\draw(20,0)--(20,1);
\node[draw,shape=circle,fill=black,scale=0.5][label=below: {-1}] at (18,0) {};
\node[draw,shape=circle,fill=white,scale=0.5][label=below: {2}] at (19,0) {};
\node[draw,shape=circle,fill=black,scale=0.5][label=below: {-3}] at (20,0) {};
\node[draw,shape=circle,fill=white,scale=0.5][label=below: {2}] at (21,0) {};
\node[draw,shape=circle,fill=black,scale=0.5][label=below: {-1}] at (22,0) {};
\node[draw,shape=circle,fill=white,scale=0.5][label=right: {2}] at (20,1) {};
\end{tikzpicture}
\caption{From left to right: $\mu^{(0)}=\lambda$, $\mu^{(1)}=s_{N_0}\mu^{(0)}$, $\mu=\mu^{(2)}=s_{N_1}\mu^{(1)}$, $-\mu=s_U\mu$, where each weight is written in the basis $\{\omega_i\}_{i=1}^n$.}
\label{fig:mu-E6}
\end{figure}

\

\noindent\textbf{Step three:} prove that $\omega_n=\pm w\mu$ for some $w\in W$.

Consider $\beta=\sum c_i\alpha_i$, written in the basis $\{\alpha_i\}_{i=1}^n$ and $\nu=\sum c_i(-1)^{d(i,n)}\omega_i=\sum_{i\in F}c_i\omega_i-\sum_{i\in U}c_i\omega_i$ written in the basis $\{\omega_i\}_{i=1}^n$ with coefficients having the same magnitude but alternating in sign. Compare
\begin{align*}
s_U\beta=&\beta-\sum_{j\in U}\langle\beta,\alpha_j^{\vee}\rangle\alpha_j\\
=&\beta-\sum_{j\in U}\left(2c_j+\sum_{i\sim j}\langle\alpha_i,\alpha_j^{\vee}\rangle c_i\right)\alpha_j\\
=&\sum_{i\in F}c_i\alpha_i-\sum_{j\in U}c_j\alpha_j-\sum_{i\sim j,i\in F,j\in U}\langle\alpha_i,\alpha_j^{\vee}\rangle c_i\alpha_j
\end{align*}
and
\begin{align*}
s_F\nu=&\nu-\sum_{i\in F}\langle\nu,\alpha_i^{\vee}\rangle\alpha_i=\nu-\sum_{i\in F}c_i\alpha_i\\
=&\sum_{i\in F}c_i\omega_i-\sum_{j\in U}c_j\omega_j-\sum_{i\in F}c_i\left(2\omega_i+\sum_{j\sim i}\langle\alpha_i,\alpha_j^{\vee}\rangle\omega_j\right)\\
=&-\sum_{i\in F}c_i\omega_i-\sum_{j\in U}c_j\omega_j-\sum_{i\sim j,i\in F,j\in U}\langle\alpha_i,\alpha_j^{\vee}\rangle c_i\omega_j
\end{align*}
now having the same coefficients on $U$ but negative signs on $F$. 

Since we can write $\xi$ as $\sum_{i\in F}b_i\alpha_i+\sum_{j\in U}b_j\alpha_j$, $\mu$ or $-\mu$ as $\sum_{i\in F}b_i\omega_i-\sum_{j\in U}b_j\omega_j$, and since we know from condition (3) that $s_Us_F\cdots s_Us_F\xi=\alpha_n$, by the reasoning above, we must have $s_Fs_U\cdots s_Fs_U\mu=\omega_n$ or $s_Fs_U\cdots s_Fs_U(-\mu)=\omega_n$. We conclude that $\omega_n$ is in the same $W$-orbit as $\lambda$.
\end{proof}

\section{Infinite fully commutative elements}
\label{sec:fully-commutative}

\subsection{Fully commutative nodes}

Given infinite reduced words $\mathbf{i}$ and $\mathbf{j}$, we say there is a \emph{braid limit} from $\mathbf{i}$ to $\mathbf{j}$, written $\mathbf{i} \to \mathbf{j}$, if there is a (possibly infinite) sequence of braid moves taking $\mathbf{i}$ to $\mathbf{j}$. Note that $\mathbf{i}\to\mathbf{j}$ does not imply $\mathbf{j}\to\mathbf{i}$ since an infinite sequence of moves might irreversibly send a letter of $\mathbf{i}$ ``to infinity" (see Example 3 of \cite{lam-pylyavskyy}).

The following proposition is a generalization of Lemma 4.6 from \cite{lam-pylyavskyy}. 

\begin{prop}
\label{prop:braid-limit-weak-order}
Let $\mathbf{i}$ and $\mathbf{j}$ be infinite reduced words.  Then $[\mathbf{j}] \leq [\mathbf{i}]$ if and only if $\mathbf{i} \to \mathbf{j}$.
\end{prop}
We omit the proof of Proposition~\ref{prop:braid-limit-weak-order} since the arguments in \cite{lam-pylyavskyy} carry over to all types.

\begin{cor}\label{cor:fully-commutative-is-minimal}
If $\mathbf{i}$ is a fully commutative infinite reduced word, then $[\mathbf{i}]$ is a minimal element in $\tilde{\mathcal{W}}$. 
\end{cor}
\begin{proof}
Since $\mathbf{i}$ is fully commutative, any braid limit $\mathbf{i}\rightarrow\mathbf{j}$ uses only commutation moves.  Since $\mathbf{i}$ is reduced, and since all parabolic subgroups of $\tilde{W}$ are finite, no single letter of $\mathbf{i}$ can move off to infinity, as it would eventually encounter another letter of the same kind. This implies that $[\mathbf{i}]=[\mathbf{j}]$, since any finite sequence of these moves does not change the inversion set. Therefore, by Proposition~\ref{prop:braid-limit-weak-order}, there does not exist $[\mathbf{j}]$ strictly smaller than $[\mathbf{i}]$ in $\tilde{\mathcal{W}}$. 
\end{proof}

\begin{lemma}\label{lem:fully-commutative-independent-of-translates}
Let $\lambda\in Q^{\vee}$. Then $[t_{\lambda}^\infty]$ is fully commutative if and only if $[t_{w\lambda}^\infty]$ is fully commutative for any $w\in W$.
\end{lemma}
\begin{proof}
We first investigate relations between explicit words of $t_{k_i\omega_i^{\vee}}^\infty$ and $t_{wk_i\omega_i^{\vee}}^\infty$. Recall that $J_i=\{s_j \in S\:|\: j\neq i\}$ and let $w=w^{J_i}w_{J_i}$ be the parabolic decomposition. Since $w_{J_i}$ fixes $\omega_i^{\vee}$, we have $wk_i\omega_i^{\vee}=w^{J_i}k_i\omega_i^{\vee}$. The inversion set $\Inv(w^{J_i})$ is contained in 
\[
\Inv(w_0^{J_i})=\{\alpha\in\Phi^+\:|\: \alpha_i\leq\alpha\},
\]
the set of all positive roots of the finite Weyl group $W$ supported on $\alpha_i$. Thus the hyperplane inversions of $w^{J_i}$ are all crossed if we move in the direction of $\omega_i^{\vee}$. Thus, we have $\Inv(w^{J_i})\subset\Inv(t_{-k_i\omega_i^{\vee}})$. As $t_{k_i\omega_i^{\vee}}=t_{-k_i\omega_i^{\vee}}^{-1}$, we can then recognize $(w^{J_i})^{-1}$ as a prefix for $t_{k_i\omega_i^{\vee}}$ and write $u=w^{J_i}t_{k_i\omega_i^{\vee}}$. In this way, we can choose reduced words for $(w^{J_i})^{-1}$ and $u$ so that $t_{k_i\omega_i^{\vee}}=(w^{J_i})^{-1}u$ and $t_{wk_i\omega_i^{\vee}}=t_{w^{J_i}k_i\omega_i^{\vee}}=w^{J_i}t_{k_i\omega_i^{\vee}}(w^{J_i})^{-1}=u(w^{J_i})^{-1}$. Therefore, both $t_{k_i\omega_i^{\vee}}^\infty$ and $t_{wk_i\omega_i^{\vee}}^\infty$ are consecutive subwords of each other. They must be both fully commutative or not fully commutative at the same time.

For the purpose of this lemma, we can without loss of generality assume that $[t_{\lambda}^\infty]$ is fully commutative. By Corollary~\ref{cor:fully-commutative-is-minimal} and Proposition~\ref{prop:minimal}, we have that $\lambda=u\omega_i^{\vee}$ for some $u\in W$ and fundamental coweight $\omega_i^{\vee}$. By our argument above, $[t_{w\lambda}^\infty]=[t_{wu\omega_i^{\vee}}^\infty]$ is fully commutative as well.
\end{proof}

Building up from Corollary~\ref{cor:fully-commutative-is-minimal}, Proposition~\ref{prop:minimal} and Lemma~\ref{lem:fully-commutative-independent-of-translates}, we see that an infinite fully commutative reduced word $[\mathbf{i}]$ must be $[t_{wk_i\omega_i^{\vee}}^\infty]$ for some $w\in W$ and some particular fundamental coweight $\omega_i^{\vee}$.
\begin{defin}\label{def:fully-commutative-node}
We say that a node $\alpha_i$ of the Dynkin diagram of $W$ is \textit{fully commutative} if $[t_{k_i\omega_i^{\vee}}^\infty]$ (or equivalently, $[t_{wk_i\omega_i^{\vee}}^\infty]$ for any $w\in W$) is fully commutative.
\end{defin}

A weight $\lambda \in P$ is \emph{minuscule} if all weights in the associated irreducible representation of the corresponding simple Lie algebra lie in the $W$-orbit of $\lambda$, and \emph{cominuscule} if $\lambda^{\vee}$ is a minuscule weight for the dual root system.  The classification of minuscule weights is well known (see, e.g. \cite{Bourbaki}). We say that a node of the Dynkin diagram is \textit{minuscule} (resp. \emph{cominuscule}) if the corresponding fundamental weight is minuscule (resp. cominuscule).

The following is our main result of the section, completely answering Problem~\ref{prob:are-minimal-fully-commutative}.

\begin{theorem}\label{thm:fully-commutative-is-minuscule}
Let $\widetilde{W}$ be any affine Weyl group, then a node is fully commutative if and only if it is minuscule or cominuscule.
\end{theorem}
\begin{proof}
We first show that a node $\alpha_i$ is fully commutative if and only if the element $w_0^{J_i} \in W$ is fully commutative.

Suppose that $\alpha_i$ is fully commutative, then $t_{k_i\omega_i^{\vee}}$ is fully commutative by definition. By the proof of Lemma~\ref{lem:fully-commutative-independent-of-translates}, $(w_0^{J_i})^{-1} \leq_R t_{k_i\omega_i^{\vee}}$, so $w_0^{J_i}$ must be fully commutative as well: if any braid move could be applied in a reduced word for $w_0^{J_i}$, then the same move could be applied in the reversed reduced word for $(w_0^{J_i})^{-1}$, but this is impossible since this reduced word is a prefix of a reduced word for the fully commutative element $t_{k_i\omega_i^{\vee}}$.

For the converse, suppose that $w_0^{J_i}$ is fully commutative. The sets of hyperplane inversions of $w_0^{J_i}$ and $t_{k_i\omega_i^{\vee}}^\infty$ are 
\begin{align*}
    I &= \{H_{\alpha,0} \: | \: \alpha_i \in \supp(\alpha)\}, \\
    I' &= \{H_{\alpha,k} \: | \: \alpha_i \in \supp(\alpha), k \geq 1\},
\end{align*}
respectively. In particular, the roots $\alpha$ labelling inversion hyperplanes for the two elements are the same. Now suppose that $t_{k_i\omega_i^{\vee}}^\infty$ is not fully commutative, so there is a reduced word and a finite prefix of the form
\[
v=w \underbrace{s_is_js_i \cdots}_{m_{ij}}.
\]
for $w \in \tilde{W}$ and $m_{ij} \geq 3$. 

Now, let $H_{\beta_1,b_1}, \ldots, H_{\beta_{m}, b_m}$ be the $m=m_{ij}$ hyperplanes which are hyperplane inversions of $v$ but not of $w$. The set $\{\beta_1, \ldots , \beta_m\} \subseteq \Phi^+$ may be computed in terms of the finite Weyl group $W$:
\[
\{\beta_1, \ldots, \beta_m\} = \{w'\alpha_i, w's_i'\alpha_j, w's_i's_j'\alpha_i, \ldots\}
\]
where for $u \in \tilde{W}$ we denote by $u'$ the image of $u$ under the projection $\tilde{W} \to W$ (so $s_i'=s_i$ for $i \neq 0$ and $s_0'$ is the non-simple reflection $s_{\xi}$ with respect to the highest root). An elementary calculation shows that the roots $\{\beta_1,\ldots, \beta_m\}$ span a two-dimensional subspace $V'$ of $V$ such that the root system $\Phi'=\Phi \cap V'$ is isomorphic to the irreducible root system of type $A_2, B_2,$ or $G_2$ according to whether $m_{ij}=3,4,$ or $6$ and that $\{\beta_1,\ldots, \beta_m\}$ are exactly the roots $\Phi^+ \cap V'$ in each case. 

Since the hyperplane inversions of $v$ are a subset of those of $t_{k_i \omega_i^{\vee}}^{\infty}$ and since the roots labelling hyperplanes in $I'$ are exactly the same as those labelling hyperplanes in $I$, we see that the inversion set of $w_0^{J_i} \in W$ also contains all of the positive roots of such a rank two subsystem. But by Proposition 4.1 of Billey--Postnikov \cite{Billey-Postnikov} this contradicts the assumption that $w_0^{J_i}$ was fully commutative, thus $t_{k_i \omega_i^{\vee}}^{\infty}$ is fully commutative.

Results of Proctor \cite{Proctor} and Theorem 6.1 of Stembridge \cite{stembridge-fullycommutative} imply that $w_0^{J_i} \in W$ is fully commutative if and only if $\alpha_i$ is minuscule or cominuscule, so the theorem is proven.
\end{proof}
 
\section{Densities and fully commutative infinite reduced words}
\label{sec:densities}
In this section, we work directly with the reduced words to give an alternative proof of the classification (Corollary~\ref{cor:fully-commutative-is-minimal} and Theorem~\ref{thm:fully-commutative-is-minuscule}) of infinite fully commutative reduced words. Throughout, let $W$ be a finite Weyl group that is not of type $A$. Therefore, its affinization $\widetilde{W}$ has an acyclic Dynkin diagram, which is crucial to our analysis.

Let $v$ be a node of the Dynkin diagram of $\widetilde{W}$ such that it connects to its neighbors by a single edge, i.e. $(s_vs_j)^3=\mathrm{id}$ if $j$ is a neighbor of $v$ and it connects to each of the connected components of $\widetilde{S}\setminus\{v\}$ in one of the following three ways described in Figure~\ref{fig:fully-commutative-branch}. We call them a type $A_m$ branch, a type $B_m$ branch and a type $D_m$ branch respectively, we say that $v$ is a \textit{branch node}.
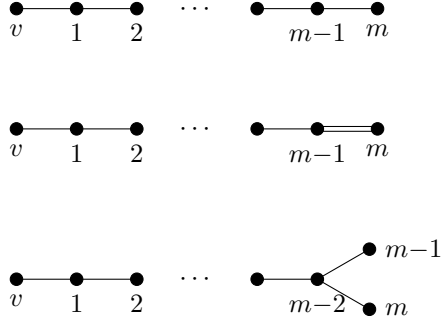
\begin{figure}[ht!]
\centering
\begin{tikzpicture}[scale=0.8]
\draw(0,0)--(2,0);
\draw(4,0)--(6,0);
\node[draw,shape=circle,fill=black,scale=0.5][label=below: {$v$}] at (0,0) {};
\node[draw,shape=circle,fill=black,scale=0.5][label=below: {$1$}] at (1,0) {};
\node[draw,shape=circle,fill=black,scale=0.5][label=below: {$2$}] at (2,0) {};
\node[draw,shape=circle,fill=black,scale=0.5][label=below: {}] at (4,0) {};
\node[draw,shape=circle,fill=black,scale=0.5][label=below: {$m{-}1$}] at (5,0) {};
\node[draw,shape=circle,fill=black,scale=0.5][label=below: {$m$}] at (6,0) {};
\node at (3,0) {$\cdots$};

\draw(0,-2)--(2,-2);
\draw(4,-2)--(5,-2);
\draw(5,-2.04)--(6,-2.04);
\draw(5,-1.96)--(6,-1.96);
\node[draw,shape=circle,fill=black,scale=0.5][label=below: {$v$}] at (0,-2) {};
\node[draw,shape=circle,fill=black,scale=0.5][label=below: {$1$}] at (1,-2) {};
\node[draw,shape=circle,fill=black,scale=0.5][label=below: {$2$}] at (2,-2) {};
\node[draw,shape=circle,fill=black,scale=0.5][label=below: {}] at (4,-2) {};
\node[draw,shape=circle,fill=black,scale=0.5][label=below: {$m{-}1$}] at (5,-2) {};
\node[draw,shape=circle,fill=black,scale=0.5][label=below: {$m$}] at (6,-2) {};
\node at (3,-2) {$\cdots$};

\draw(0,-4.5)--(2,-4.5);
\draw(4,-4.5)--(5,-4.5);
\draw(5.866,-4)--(5,-4.5)--(5.866,-5);
\node[draw,shape=circle,fill=black,scale=0.5][label=below: {$v$}] at (0,-4.5) {};
\node[draw,shape=circle,fill=black,scale=0.5][label=below: {$1$}] at (1,-4.5) {};
\node[draw,shape=circle,fill=black,scale=0.5][label=below: {$2$}] at (2,-4.5) {};
\node[draw,shape=circle,fill=black,scale=0.5][label=below: {}] at (4,-4.5) {};
\node[draw,shape=circle,fill=black,scale=0.5][label=below: {$m{-}2$}] at (5,-4.5) {};
\node[draw,shape=circle,fill=black,scale=0.5][label=right: {$m{-}1$}] at (5.866,-4) {};
\node[draw,shape=circle,fill=black,scale=0.5][label=right: {$m$}] at (5.866,-5) {};
\node at (3,-4.5) {$\cdots$};
\end{tikzpicture}
\caption{Type $A$ branch, type $B$ branch, and type $D$ branch connected to $v$.}
\label{fig:fully-commutative-branch}
\end{figure}

Notice that such a branch node $v$ does not exist for types $\widetilde{B_3}$, $\widetilde{C_3}$, $\widetilde{F_4}$ and $\widetilde{G_2}$. We will not be concerned about these types for the section since these small cases are easy to check by hand. 

Let $J_1,\ldots,J_b$ be the connected components of $\widetilde{S}\setminus\{v\}$. Let $[\mathbf{i}]$ be a class of fully commutative infinite reduced words. By identifying the simple generators $s_v$ in $\mathbf{i}$ and separating the infinite words correspondingly into blocks, we can write $$\mathbf{i}=(w_{J_1}^{(0)}w_{J_2}^{(0)}\cdots w_{J_b}^{(0)})s_v(w_{J_1}^{(1)}\cdots w_{J_b}^{(1)})s_v\cdots s_v(w_{J_1}^{(p)}\cdots w_{J_b}^{(p)})s_v\cdots$$
where $w_{J_k}^{(p)}$ is in the parabolic subgroup generated by $J_k$. Notice that two generators in different connected components of $\widetilde{S}\setminus\{v\}$ commute. So $w_{J}^{(p)}$ commutes with $w_{J'}^{(p')}$ for $J\neq J'$.

Since commutation moves are allowed, each $w_{J_k}^{(p)}$ is not well-defined for $[\mathbf{i}]$ and is only well-defined for a particular reduced word. However, we can still make the following definitions. For each $J\in\{J_1,\ldots,J_b\}$ and $p\geq0$, we define
\[
d([\mathbf{i}],J)_p=
\begin{cases}
0,&\text{ if commutation moves can be applied so that } w^{(p)}_J=\mathrm{id}\\
2,&\text{ if commutation moves can be applied so that } w^{(p)}_{J'}=\mathrm{id}\text{ for all }J'\neq J\\
1,&\text{ otherwise}
\end{cases}.
\]
We see that $d([\mathbf{i}],J)_p$ is well-defined because the first two situations above cannot happen simultaneously, which would mean that certain commutations moves can be applied so that $W_J^{(p)}=\mathrm{id}$ for all $J$ and we cannot have a reduced word in this case. 

The parameter $d([\mathbf{i}],J)$, which is an infinite vector, can be intuitively thought of as indicating the ``density" of the branch $J$ in the reduced word $[\mathbf{i}]$. 
\begin{lemma}\label{lem:density-at-least-2}
For every $p\geq1$, $d([\mathbf{i}],J_1)_p+\cdots+d([\mathbf{i}],J_b)_p\geq2$.
\end{lemma}
\begin{proof}
This is immediate from the definition of $d([\mathbf{i},J])$. Notice that we can never apply commutation moves so that all of $w_J^{(p)}$'s become the identity, since that would imply $[\mathbf{i}]$ is not reduced with two consecutive $s_v$'s. If commutation moves can be applied so that only one of $w_J^{(p)}$'s is not the identity, then $d([\mathbf{i}],J)=2$ and if that cannot happen, we must have $d([\mathbf{i}],J)=1$ and $d([\mathbf{i}],J')=1$ for some $J\neq J'$. So we are done.
\end{proof}
Lemma~\ref{lem:density-at-least-2} is saying that we need a total ``density" of at least 2. We then show that this density is in fact very nontrivial to achieve for different types of branches.

\begin{lemma}\label{lem:density-typeA}
Let $J$ be a type $A_m$ branch labeled as in Figure~\ref{fig:fully-commutative-branch} and $[\mathbf{i}]$ be fully commutative. With notations as above, we then have $d([\mathbf{i}],J)_p\leq1$ for all $p$. Moreover, if $d([\mathbf{i}],J)_{p+1}=\cdots=d([\mathbf{i}],J)_{p+m}=1$, then commutation moves can be applied so that $w_J^{(p+k)}=s_ks_{k-1}\cdots s_1$ for $k=1,\ldots,m$ and $d([\mathbf{i}],J)_{p+m+1}=0.$
\end{lemma}
\begin{proof}
To show that $d([\mathbf{i}],J)_p\neq2$, it suffices to show that there are no fully commutative elements of the form $s_vw_Js_v$ where $w_J\in W_J$. Assume the opposite and pick a reduced word for $w_J=s_{i_1}\cdots s_{i_\ell}$. If the word does not start with $s_{i_1}=s_1$, then we can use commutation move to move $s_{i_1}$ to the left of the first $s_v$ and argue with $s_vs_{i_2}\cdots s_{i_\ell}s_v$. So we can similarly assume without loss of generality that $w_J=s_1s_2\cdots s_ku$ where $u$ is either the identity or starts with some $s_a$ with $a\leq k-1$. But if $u$ starts with $s_a$ with $a\leq k-1$, we can move $s_a$ to the left to obtain $s_as_{a+1}s_a$ and similarly if $u=\mathrm{id}$, we can move the $s_v$ at the end to obtain $s_vs_1s_v$. This contradicts $s_vw_Js_v$ being fully commutative.

Next, we assume that $d([\mathbf{i}],J)_{p+1}=\cdots=d([\mathbf{i}],J)_{p+m}=1$ and investigate what can $w_J^{(p+1)},\ldots,w_J^{(p+m)}$ be. We use induction to show that we can apply commutation moves so that $w_J^{(p+k)}$ becomes $s_k\cdots s_1$ for $k=1,\ldots,m$. The base case is $k=1$. If $w_J^{(p+1)}=s_{i_1}\cdots s_{i_\ell}$ does not start with $s_{i_1}=s_1$, we can use commutation moves to move $s_{i_1}$ across the $s_v$ on the left. Thus, assume $s_{i_1}=s_1$. Similarly we can assume that $s_{i_2}=s_2,\ldots,s_{i_k}=s_k$ for some $k\geq1$ and $i_{k+1}<k+1$. Then we create a three term Coxeter move $s_as_{a+1}s_a$ if we move $i_{k+1}$ to the left. As a result, $w_J^{(p+1)}$ has now become $s_1s_2\cdots s_k$ for some $k\geq1$, we can then use commutation moves to move $s_2\cdots s_k$ across the $s_v$ on the right. In this way, $w_J^{(p+1)}$ becomes $s_1$ as desired.

Now assume the induction hypothesis that $w_J^{(p+a)}=s_a\cdots s_1$ for $a\leq k$ and consider $w_J^{(p+k+1)}$. By the same reasoning of moving simple generators across the $s_v$ on the right, we can assume that $w_J^{(p+k+1)}$ end with $s_a s_{a-1}\cdots s_2s_1$. In fact, we can assume that $w_J^{(p+k+1)}$ has become $s_a\cdots s_1$ since any additional $s_j$ immediately to the left of $s_a$ will either be moved across the $s_v$ on the right if $j\geq a+2$, or create a 3-term Coxeter move $s_js_{j+1}s_j$ if $j\leq a-1$ (and if $j=a+1$ then this sequence just gets extended and clearly $j\neq 1$). If $a=1$, then we have a Coxeter move $s_vs_1s_v$ available and if $2\leq a\leq k$, we can move $s_a$ across the $s_v$ into $w_J^{(p+k)}=s_k\cdots s_1$ to create a Coxeter move $s_as_{a-1}s_1$. And if $a\geq k+2$, we can move $s_a$ all the way to the left into $w_J^{(p)}$ so that it's out of consideration. As a result, we must have $a=k+1$ and we have used commutation moves to let $w_J^{(p+k+1)}$ become $s_{k+1}\cdots s_1$ as desired. The induction step goes through.

Finally, when $d([\mathbf{i}],J)_{p+1}=\cdots=d([\mathbf{i}],J)_{p+m}=1$, we have seen that commutation moves so that $w_J^{(p+k)}=s_ks_{k-1}\cdots s_1$ for $k=1,\ldots,m$. Now for $w_J^{(p+m+1)}$, if it is not the identity, with the same reasoning as above, we assume it is $s_a\cdots s_1$, for some $a$. But we either get a Coxeter move $s_vs_1s_1$ or $s_as_{a-1}s_a$ by moving $s_a$ into $w_J^{(p+m)}$. Therefore, $d([\mathbf{i}],J)=0$.
\end{proof}

Intuitively, Lemma~\ref{lem:density-typeA} is saying that a branch of type $A_m$ has a density at most $m/(m+1)$. We continue such analysis for type $B_m$ branches and type $D_m$ branches. 
\begin{lemma}\label{lem:density-typeB}
Let $J$ be a type $B_m$ branch labeled as in Figure~\ref{fig:fully-commutative-branch} and $[\mathbf{i}]$ be fully commutative. If $d([\mathbf{i}],J)_p=2$, then $w_J^{(p)}=s_1s_2\cdots s_m\cdots s_1$, and $w_J^{(p-1)}=w_J^{(p+1)}=\mathrm{id}$. On the other hand, if the sequence $d([\mathbf{i}],J)$ is eventually 1, then commutation moves can be applied so that $w_J^{(p)}=s_ms_{m-1}\cdots s_2s_1$ for all sufficiently large $p$. 
\end{lemma}
\begin{proof}
First assume that $d([\mathbf{i}],J)_p=2$. This means that commutation moves can be applied so that $s_vw_J^{(p)}s_v$ is fully commutative. We can assume that a reduced word of $w_J^{(p)}$ starts with $s_1s_2\cdots s_a$ for some $a\geq1$, since otherwise commutation moves can be applied to move simple generators across the left $s_v$. At the same time we cannot have $w_J^{(p)}=s_1s_2\cdots s_a$ since a Coxeter move $s_vs_1s_v$ would be available. Say $w_J^{(p)}$ starts with $s_1s_2\cdots s_as_k$ with $k<a$, In order for there not to be a Coxeter move of the form $s_ks_{k+1}s_k$, we must have $a=m$ and $k=m-1$. Likewise, $w_J^{(p)}$ cannot be $s_1s_2\cdots s_ms_{m-1}$ since $s_vs_1s_v$ is an available Coxeter move and the only $s_k$ that will not cause a Coxeter move to arrive when appended after $s_1s_2\cdots s_ms_{m-1}$ is $s_{m-2}$. Continue this process, we conclude that $w_J^{(p)}$ must start with $s_1s_2\cdots s_m\cdots s_1$. At this point, we see that no more simple generators can be appended so $w_J^{(p)}$ becomes $s_1s_2\cdots s_m\cdots s_1$ as desired. Moreover, if $w_J^{(p+1)}\neq\mathrm{id}$, let $s_k$ be a prefix of $w_J^{(p+1)}$. If $k\geq2$, we can move $s_k$ across $s_v$ into $w_J^{(p)}$ to create a Coxeter move $s_ks_{k-1}s_k$ when $k\neq m$ and a Coxeter move $s_{m-1}s_ms_{m-1}s_m$ when $k=m$; and if $k=1$, we have Coxeter move $s_vs_1s_v$. This means $w_J^{(p+1)}=\mathrm{id}$ and by symmetry, $w_J^{(p-1)}=\mathrm{id}$ as well. 

Notice that in the lemma statement, we say that in the case of $d([\mathbf{i}],J)=2$, $w_J^{(p)}=s_1\cdots s_m\cdots s_1$ without commutation moves needed. This is because no commutation moves can be applied to adjust $w_J^{(p-1)}=w_J^{(p+1)}=\mathrm{id}$ and $w_J^{(p)}=s_1\cdots s_m\cdots s_1$ relative to each other.

Moving on to the next case, we notice that if $d([\mathbf{i}],J)$ is eventually positive, then this sequence eventually takes on the value 1 since any $d([\mathbf{i}],J)_p$ being 2 results in the next term being 0. Assume that $d([\mathbf{i}],J)_{p+1}=d([\mathbf{i}],J)_{p+2}=\cdots=1$. We then use induction to show that commutation moves can be applied so that $w_J^{(p+k)}=s_k\cdots s_1$ for $k=1,\ldots,m$. The base case is $k=1$. Assume without loss of generality that $w_J^{(p+1)}$ starts with $s_1s_2\cdots s_a$ for some $a\geq1$. If $w_J^{(p+1)}=s_1s_2\cdots s_a$, we can use commutation moves to move $s_2\cdots s_a$ into $w_J^{(p+2)}$ so that $w_J^{(p)}$ becomes $s_1$. If there are more generators after $s_1s_2\cdots s_a$, the only $s_k$'s that cannot be moved past the $s_v$ on the left or create a Coxeter move is for $a=m$ and $k=m-1$. Arguing analogously, we end up with $w_J^{(p+1)}=s_1s_2\cdots s_m\cdots s_k$ for some $1\leq k\leq m$ when we see that no more generators can be added to avoid Coxeter moves. If $k=1$, from the case addressed in the last paragraph, we must have $d([\mathbf{i}],J)_{p+2}=0$ which is impossible and if $k>1$, we can move $s_2\cdots s_m\cdots s_k$ across the $s_v$ on the right into $w_J^{(p+2)}$. In this way, we make sure that $w_J^{(p+1)}$ becomes $s_1$. For the inductive step, we assume that $w_J^{(p+j)}$ has become $s_j\cdots s_1$ for $j\leq k$, $j<m$ and we consider $w_J^{(p+k+1)}$, which is not the identity. Assume without loss of generality (by moving generators across the $s_v$ on the right) that $w_J^{(p+k+1)}$ ends with $s_as_{a-1}\cdots s_1$ for some maximal $a\geq1$. If $w_J^{(p+k+1)}\neq s_a\cdots s_1$, then as above, we must have $a=m$, $a'<m$ and $w_J^{(p+k+1)}=s_{a'}\cdots s_m\cdots s_1$. But then $s_{a'}$ will create a Coxeter move with $w_J^{(p+k)}$, contradiction. Thus, we obtain $w_J^{(p+k+1)}=s_as_{a-1}\cdots s_1$ for some $a\geq1$. If $a\leq k$, then moving $s_a$ across the $s_v$ on the left into $w_J^{(p+k)}$ creates a Coxeter move and if $a\geq k+2$, we can move $s_a$ all the way into $w_J^{(p)}$ by induction hypothesis. In the end, we are left with $w_J^{(p+k+1)}=s_{k+1}\cdots s_1$ as desired. So the induction step goes through.

Now that we have $w_J^{(p+m)}=s_m\cdots s_1$, we show that assuming $d([\mathbf{i}],J)_{p+m+1}>0$, we can use commutation moves so that $d([\mathbf{i}],J)_{k}=s_m\cdots s_1$ for all $k\geq p+m$. We will be using commutation moves only to move the generators to the right so all previous $w_J^{(j)}$'s will be preserved, for $j\leq p+m$. Similarly assume that $w_J^{(p+m+1)}$ ends with $s_a\cdots s_1$ with maximal $a$. If $w_J^{(p+m+1)}\neq s_a\cdots s_1$, then let $s_{a'}$ be the simple generator left of $s_a$ with $a'<a+1$. Then $a'\neq a$ so we will have a Coxeter move $s_{a'}s_{a'+1}s_{a'}$, contradiction. So $w_J^{(p+m+1)}=s_a\cdots s_1$, and $a$ must be $m$ to avoid Coxeter moves with $w_J^{(p+m)}=s_m\cdots s_1$. This concludes the proof.
\end{proof}

\begin{lemma}\label{lem:density-typeD}
Let $J$ be a type $D_m$ branch labeled as in Figure~\ref{fig:fully-commutative-branch} and $[\mathbf{i}]$ be fully commutative. If $d([\mathbf{i}],J)_p=2$, then $w_J^{(p)}=s_1s_2\cdots s_{m-2}s_{m-1}s_{m}s_{m-2}\cdots s_1$, and $w_J^{(p-1)}=w_J^{(p)}=\mathrm{id}$. On the other hand, if the sequence $d([\mathbf{i}],J)$ is eventually 1, then commutation moves can be applied so that $w_J^{(2p+\epsilon)}=s_ms_{m-2}\cdots s_1$ and $w_J^{(2p+\epsilon+1)}=s_{m-1}s_{m-2}\cdots s_1$ for some $\epsilon\in\{0,1\}$ and sufficiently large $p$.
\end{lemma}
The proof of Lemma~\ref{lem:density-typeD} is exactly the same as Lemma~\ref{lem:density-typeB} so we skip the proof.

Intuitively, Lemma~\ref{lem:density-typeB} and Lemma~\ref{lem:density-typeD} is saying that a type $B_m$ branch and a type $D_m$ branch both have ``density" at most 1. Surprisingly (or not surprisingly), the sum of the ``densities" of all the branches provided by Lemma~\ref{lem:density-typeA}, Lemma~\ref{lem:density-typeB} and Lemma~\ref{lem:density-typeD} are bounded above by exactly 2, which is also needed for being fully commutative via Lemma~\ref{lem:density-at-least-2}. For example, we can choose a branch point $v$ so that the type $\widetilde{E_8}$ Dynkin diagram has three branches of type $A_1$, $A_2$ and $A_5$, whose ``densities" are bounded via Lemma~\ref{lem:density-typeA} by $1/2+2/3+5/6=2$.

\begin{theorem}\label{thm:explicit-fully-commutative}
The following is the full list of all fully commutative infinite words:
\begin{table}[ht!]
\centering
\begin{tabular}{c|c|c}
type & reduced words & coweight \\\hline
$\widetilde{A_n}$ &  $c^\infty$ for any standard Coxeter element $c$ &  \\\hline
$\widetilde{B_n}$ & $w(s_2s_0s_1s_2s_3s_4\cdots s_ns_{n-1}\cdots s_4s_3)^\infty$ & $\omega_1^{\vee}$\\
& $w(s_2s_0s_3s_4\cdots s_ns_2s_1s_3s_4\cdots s_n)^\infty$ & $\omega_n^{\vee}$\\\hline
$\widetilde{C_n}$ & $w(s_1s_0s_1s_2s_3s_4\cdots s_ns_{n-1}\cdots s_3s_2)^\infty$ & $\omega_1^{\vee}$\\
& $w(s_0s_1\cdots s_{n-1}s_n)^\infty$ & $\omega_n^{\vee}$ \\\hline
$\widetilde{D_n}$ & $w(s_2s_0s_1s_2s_3s_3\cdots s_{n-2}s_{n-1}s_ns_{n-2}s_{n-3}\cdots s_4s_3)^\infty$ & $\omega_1^{\vee}$ \\
& $w(s_2s_0s_3s_4\cdots s_{n-2}s_{n-1}s_2s_1s_3s_4\cdots s_{n-2}s_n)^\infty$ & $\omega_{n-1}^{\vee}$\\
& $w(s_2s_1s_3s_4\cdots s_{n-2}s_{n-1}s_2s_0s_3s_4\cdots s_{n-2}s_n)^\infty$ & $\omega_{n}^{\vee}$\\\hline
$\widetilde{E_6}$ & $w(s_3s_2s_5s_4s_3s_1s_2s_6s_3s_4s_0s_6)^\infty$ & $\omega_6^{\vee}$\\
& $w(s_3s_2s_0s_6s_3s_1s_2s_4s_3s_5s_4s_6)^\infty$ & $\omega_1^{\vee}$ \\\hline
$\widetilde{E_7}$ & $w(s_3s_7s_5s_4s_3s_2s_6s_5s_4s_3s_7s_1s_2s_3s_0s_1s_2s_4)^\infty$ & $\omega_6^\vee$\\\hline
\end{tabular}
\caption{Fully commutative infinite reduced words}
\label{tab:fully-commutative}
\end{table}
where the labels of the generators are seen in Figure~\ref{fig:main-fig} and $w\in \widetilde{W}$ so that we can describe all fully commutative words.
\end{theorem}
\begin{proof}
The theorem is a direct consequence of Lemma~\ref{lem:density-typeA}, Lemma~\ref{lem:density-typeB} and Lemma~\ref{lem:density-typeD}. The arguments are separate but largely similar in all types so we will discuss selective types to avoid extra tediousness. 

For $\widetilde{E_7}$ as labeled in Figure~\ref{fig:main-fig}, the branch point is $v=3$. Let $J_1=\{7\}$, $J_2=\{0,1,2\}$ and $J_3=\{4,5,6\}$. By Lemma~\ref{lem:density-typeA}, say  $w_{J_3}^{(p+1)}=\mathrm{id}$. Then by Lemma~\ref{lem:density-at-least-2}, we must have $d([\mathbf{i}],J_1)=d([\mathbf{i}],J_2)=1$. Then $w_{J_1}^{(p+1)}=s_7$ and $d([\mathbf{i}],J_1)_{p+2}=0$. This means $d([\mathbf{i}],J_2)=d([\mathbf{i}],J_3)=1$. If $d([\mathbf{i}],J_2)_{p+3}=1$, then $d([\mathbf{i}],J_2)_{p+4}=0$, $d([\mathbf{i}],J_1)_{p+4}=d([\mathbf{i}],J_3)=1$, $d([\mathbf{i}],J_1)_{p+3}=d([\mathbf{i}],J_1)_{p+5}=0$, $d([\mathbf{i}],J_3)_{p+3}=1$, $d(\mathbf{[i]},J_3)_{p+5}=0$, contradicting Lemma~\ref{lem:density-at-least-2} at $p+5$. Thus, $d([\mathbf{i}],J_2)_{p+3}=0$. And $d([\mathbf{i}],J_1)_{p+3}=d([\mathbf{i}],J_3)_{p+3}=1$, $d([\mathbf{i}],J_1)_{p+4}=0$, $d([\mathbf{i}],J_3)_{p+4}=1$. One can also fill in the specific reduced word by Lemma~\ref{lem:density-typeA} to complete one full period (see Table~\ref{tab:E7-fully-commutative}). Continue this process and we obtain the explicit reduced word shown in Table~\ref{tab:fully-commutative}.
\begin{table}[ht!]
\centering
\begin{tabular}{c|cccc|c}
& $p+1$ & $p+2$ & $p+3$ & $p+4$ & $\cdots$ \\\hline
$w_{J_1}$ & $s_7$ & $\mathrm{id}$ & $s_7$ & $\mathrm{id}$ & $s_7$ \\
$w_{J_2}$ & $s_1s_2$ & $s_0s_1s_2$ & $\mathrm{id}$ & $s_2$ & $s_1s_2$ \\
$w_{J_3}$ & $\mathrm{id}$ & $s_4$ & $s_5s_4$ & $s_6s_5s_4$ & $\mathrm{id}$
\end{tabular}
\caption{Fully commutative infinite words for $\widetilde{E_7}$}
\label{tab:E7-fully-commutative}
\end{table}
The same argument for $\widetilde{E_7}$ also works for $\widetilde{E_6}$, $\widetilde{D_4}$ and can eliminate all possibilities in $\widetilde{E_8}$.

For $\widetilde{C_n}$, $n\geq4$ as labeled in Figure~\ref{fig:main-fig}, we can choose any branch point $v$ from $\{2,3,\ldots,n-2\}$. Once a branch point $v$ is fixed, the two branches $J_1=\{0,1,\ldots,v-1\}$ and $J_2=\{v+1,\ldots,n\}$ are both type $B$ branches. If $d([\mathbf{i}],J_1)_p=0$ for some $p$, then we necessarily have $d([\mathbf{i}],J_2)_p=2$ with $w_{J_2}^{(p)}=s_{v+1}\cdots s_n\cdots s_{v+1}$ given by Lemma~\ref{lem:density-typeB}. Consequently, $d([\mathbf{i}],J_2)_{p+1}=0$, $d(\mathbf{i},J_1)_{p+1}=2$ with $w_{J_2}^{(p+1)}=s_{v-1}\cdots s_1s_0s_1\cdots s_{v-1}$. Continuing this process, we see that $[\mathbf{i}]=w(s_0s_1\cdots s_n\cdots s_1)^\infty$, which corresponds to the coweight $\omega_1^{\vee}$. As for the other case, if $d([\mathbf{i}],J_1)_p>0$ for all $p$, then we must have $d([\mathbf{i}],J_1)_p=d([\mathbf{i}],J_2)_p=1$ for all $p$. By Lemma~\ref{lem:density-typeB}, we can use commutation moves so that $w_{J_1}^{(p)}=s_1\cdots s_{v-1}$ and $w_{J_2}^{(p)}=s_{v+1}\cdots s_{n}$. Now $[\mathbf{i}]=wc^\infty$ for a standard Coxeter element $c$, which corresponds to the coweight $\omega_n^{\vee}$. The same argument also works for $\widetilde{D_n}$ and $\widetilde{B_n}$, together with Lemma~\ref{lem:density-typeD}.
\end{proof}

\section*{Acknowledgements}
We are grateful to Alex Postnikov for helpful suggestions and especially to Thomas Lam for introducing us to these problems and for sharing many ideas.

\bibliographystyle{plain}
\bibliography{main}
\end{document}